\documentclass[11pt, a4paper]{article}
\usepackage[ddmmyyyy]{datetime}
\usepackage{float}
\usepackage{amsfonts}
\usepackage{mathrsfs}
\usepackage{amssymb}
\usepackage{amsmath}
\usepackage{amsthm}
\usepackage{appendix}
\usepackage{enumitem}
\usepackage{hyperref}
\usepackage[usenames,dvipsnames]{xcolor}
\usepackage{graphicx}
\usepackage{booktabs}
\usepackage[font=small,labelfont=bf]{caption}
\usepackage{subcaption}
\usepackage{geometry}

\newcommand\footnotenomark[1]{%
  \begingroup
  \renewcommand\thefootnote{}\footnote{#1}%
  \addtocounter{footnote}{-1}%
  \endgroup
}

\hypersetup{
    colorlinks=true,       
    linkcolor=blue,          
    citecolor=blue,        
    urlcolor=blue           
}

\numberwithin{equation}{section}

\newtheorem{theorem}{Theorem}
\numberwithin{theorem}{section}

\newtheorem*{theorem*}{Theorem}

\newtheorem{corollary}[theorem]{Corollary}

\newtheorem*{corollary*}{Corollary}

\newtheorem*{proposition*}{Proposition}

\newtheorem{lemma}[theorem]{Lemma}

\newtheorem*{lemma*}{Lemma}

\newtheorem*{conjecture*}{Conjecture}

\newtheorem*{remark}{Remark}

\newtheorem{definition}[theorem]{Definition}

\newtheorem*{definition*}{Definition}

\newtheorem*{algorithm*}{Algorithm}

\newtheorem*{example*}{Example}

\renewcommand{\thefootnote}{\fnsymbol{footnote}}	

\newcommand{\Rone}{ 
    \mathbb{R}
}

\newcommand{\Rd}[1][d]{ 
    \mathbb{R}^{#1}
}

\newcommand{\Rplus}{ 
    \mathbb{R}_{+}
}

\newcommand{\Zp}{ 
    \mathbb{Z}_{+}
}

\newcommand{\Nz}{ 
    \mathbb{N}_{0}
}

\newcommand{\PT}[1]{\mathbf{#1}} 

\newcommand{\Lshift}[1][L]{   
  \widetilde{#1}
}

\newcommand{\Lhat}{   
  \ellshift[L]
}

\newcommand{\ellshift}[1][\ell]{   
  \widehat{#1}
}

\newcommand{\Vd}[2][d]{   
  V_{#2}^{#1}
}

\newcommand{\vd}[2][d]{
  v_{#2}^{#1}} 

\newcommand{\vdh}[2][]{ 
  v_{#2}^{#1}
}

\newcommand{\Vdh}[2][]{ 
  V_{#2}^{#1}
}

\newcommand{\vdl}[2][d,\delta]{ 
  v_{#2}^{#1}
}

\newcommand{\Vdl}[2][d,\delta]{
  V_{#2}^{#1}} 

\newcommand{\Vdhl}[2][d,\delta]{ 
  V_{#2}^{#1}
}

\newcommand{\vab}[2][\alpha,\beta]{
  v_{#2}^{(#1)}} 

\newcommand{\Vab}[2][\alpha,\beta] {
  \mathcal{V}_{#2}^{(#1)}} 

\newcommand{\Vabh}[2][(\alpha,\beta)]{
  V_{#2}^{#1}
}

\newcommand{\vabh}[2][(\alpha,\beta)]{
  v_{#2}^{#1}
}

\newcommand{\fis}{ 
    \kappa
}

\newcommand{\fil}{ 
    g
}

\newcommand{\FA}[3][\ellshift]{
  F^{(#2)}_{\alpha,\beta}(#1,#3)} 

\newcommand{\FB}[2][\alpha,\beta]{
  F^{(2)}_{#1}(#2)} 

\newcommand{\LBo}[1][]{\Delta^{*}_{#1} 
}

\newcommand{\Jcb}[2][\alpha,\beta]{
    P^{(#1)}_{#2}
}

\newcommand{\NGegen}[2][d+1]{   
  P^{(#1)}_{#2}
}

\newcommand{\Jw}[1][\alpha,\beta]{ 
w_{#1}
}

\newcommand{\dist}[1] { 
  \mathrm{dist}(#1)}

\newcommand{\scap}[1] { 
  \mathcal{C}\left(#1\right)}

\newcommand{\sph}[1]{\mathbb{S}^{#1} 
}

\newcommand{\shSp}[2][d]{ 
\mathcal{H}_{#2}(\sph{#1})
}

\newcommand{\Lp}[3][]{ 
\mathbb{L}_{#2}(\sph{#3}#1)
}

\newcommand{\sob}[3]{ 
\mathbb{W}_{#1}^{#2}(\sph{#3})}

\newcommand{\CkR}[1][\fis]{ 
    C^{#1}(\mathbb{R}_{+})
}

\newcommand{\Lpw}[2][\Jw]{ 
\mathbb{L}_{#2}(#1)
}

\newcommand{\prj}[2][]{ 
\mathrm{Y}_{#2}^{#1}}

\newcommand{\norm}[2]{ 
\|#1\|_{#2}}

\newcommand{\normb}[2]{ 
\big\|#1\big\|_{#2}}

\newcommand{\normau}[2]{ 
\left\|#1\right\|_{#2}}

\newcommand{\Tsph}[2][]{ 
\mathbf{T}_{#2}^{#1}}

\newcommand{\Tx}[1][\PT{x}]{ 
    \mathcal{T}_{#1}
}

\newcommand{\InnerL}[2][d]{ 
\left(#2\right)_{\Lp{2}{#1}}
}

\newcommand{\InnerLb}[2][d]{ 
\bigl(#2\bigr)_{\Lp{2}{#1}}}

\newcommand{\IntDiff}[2][d]{    
\:\mathrm{d}\sigma_{#1}(\PT{#2})
}

\newcommand{\IntD}[1]{ 
\:\mathrm{d}{#1}
}

\newcommand{\Diff}[2][t]{ 
\ifthenelse{\equal{#2}{1}}{\frac{\mathrm{d}}{\mathrm{d}#1}}{
\left(\frac{\mathrm{d}}{\mathrm{d}#1}\right)^{#2}}
}

\newcommand{\supp}{ 
\mathrm{supp}\:
}

\newcommand{\floor}[1]{ 
\left\lfloor #1 \right\rfloor
}

\newcommand{\bigo}[2] { 
    \mathcal{O}_{#1}\left(#2\right)
}

\newcommand{\Ak}[1]{ 
    A_{#1}
}

\newcommand{\FDiff}[2][\ell]{\overrightarrow{\Delta}_{#1}^{#2} 
}

\newcommand{\thetan}[1][\theta,L]{ 
\widetilde{u}\left(#1\right) }

\newcommand{\untheta}[2][L]{    
\bigl((#1+\tfrac{d}{2})#2-\tfrac{d-1}{4}\pi\bigr)
}

\newcommand{\Jcoe}[2][(\alpha,\beta)]{ 
    M_{#2}^{#1}
}

\newcommand{\fractional}[1]{ 
\left\langle #1 \right\rangle
}

\newcommand{\Ck}[2]{ 
    C^{#1}(#2)
}

\newcommand{\ddelta}{ 
    \boldsymbol{\delta}
}

\newcommand{\arrowDelta}[2][\ell]{ 
    \overrightarrow{\Delta}_{#1}^{#2}\hspace{0.3mm}
}
\newcommand{\modu}[2][\Lp{p}{d}]{ 
\omega_{2}\left(#2\right)_{#1}}

\newcommand{\modub}[2][\Lp{p}{d}]{ 
\omega_{2}\brb{#2}_{#1}}

\newcommand{\Kf}[2][\Lp{p}{d}]{ 
K\left(#2\right)_{#1}
}

\newcommand{\InnerLab}[2][\alpha,\beta]{ 
\left(#2\right)_{#1}
}

\newcommand{\coZa}[1]{ 
    H_{#1}
}

\newcommand{\coZb}[1]{ 
    H'_{#1}
}

\newcommand{\afiJcbu}[1]{ 
    u_{#1}
}

\newcommand{\brb}[1]{ 
 \bigl(#1\bigr)
}

\begin{document}

\title{Riemann localisation on the sphere\thanks{\today}}

\author{Yu Guang Wang\footnote{Corresponding author.}\footnotenomark{Emails: \tt yuguang.e.wang@gmail.com,\; i.sloan@unsw.edu.au,\; R.Womersley@unsw.edu.au}\;\footnote{Present address: Department of Mathematics, City University of Hong Kong, Hong Kong.}
\quad\; Ian H. Sloan
\quad\; Robert S. Womersley
\\[5mm] School of Mathematics and Statistics
\\[1mm] University of New South Wales, Australia}

\date{}

\maketitle

\begin{center}

\begin{abstract}
This paper first shows that the Riemann localisation property holds for the Fourier-Laplace series partial sum for sufficiently smooth functions on the two-dimensional sphere, but does not hold for spheres of higher dimension. By Riemann localisation on the sphere $\sph{d}\subset\Rd[d+1]$, $d\ge2$, we mean that for a suitable subset $X$ of $\Lp{p}{d}$, $1\le p\le \infty$, the $\mathbb{L}_{p}$-norm of the Fourier local convolution of $f\in X$ converges to zero as the degree goes to infinity. The Fourier local convolution of $f$ at $\PT{x}\in\sph{d}$ is the Fourier convolution with a modified version of $f$ obtained by replacing values of $f$ by zero on a neighbourhood of $\PT{x}$. The failure of Riemann localisation for $d>2$ can be overcome by considering a filtered version: we prove that for a sphere of any dimension and sufficiently smooth filter the corresponding local convolution always has the Riemann localisation property. Key tools are asymptotic estimates of the Fourier and filtered kernels.


\end{abstract}

\end{center}

\normalsize
\section{Introduction}\label{sec:RFoS.intro}
The well known Riemann-Lebesgue lemma (see, for example, \cite[Theorem~1.4, p.~80]{StSh2003}) states that the $L$th Fourier coefficient of an integrable function on the circle $\mathbb{S}^{1}$ approaches zero as $L$ approaches $\infty$. As a direct consequence (as explained below), the Riemann localisation property holds, meaning that for an integrable $2\pi$-periodic function $f$ that vanishes on an open interval, the $L$th partial sum of the Fourier series approaches zero as $L$ approaches $\infty$ at every point of that open interval.  An equivalent statement is that the Fourier local convolution of an integrable $2\pi$-periodic function on the circle (where the local convolution at $\theta$ is the convolution of the $L$th Dirichlet kernel with the function modified by replacing by zero its values in a neighborhood of $\theta$) approaches zero  as the degree of the Dirichlet kernel approaches $\infty$.

This paper extends the notion of Riemann localisation to spheres $\sph{d}\subset\mathbb{R}^{d+1}$ of arbitrary dimensions $d\ge2$. We define the Fourier local convolution on $\sph{d}$ and obtain tight upper and lower bounds of the $\mathbb{L}_{p}$-norm of the Fourier local convolution for functions in Sobolev spaces.  We shall see that Riemann localisation holds for sufficiently smooth functions on $\sph{2}$, but does not hold at all for spheres $\sph{d}$ with $d>2$.
We then define a filtered version of the Fourier convolution, and prove that the filtered convolution has the Riemann localisation property for a sphere of any dimension and filter of sufficient smoothness.

In more detail, for the circle $\sph{1}$, the Fourier partial sum of order $L\ge1$ for $f\in \Lp{1}{1}$ may be written as
\begin{equation*}
  \Vd[]{L}(f;\theta):=\Vd[1]{L}(f;\theta):=\frac{1}{2\pi}\int_{-\pi}^{\pi}\vd[]{L}(\theta-\phi)f(\phi)\:\mathrm{d}\phi=\frac{1}{2\pi}\int_{-\pi}^{\pi}\vd[]{L}(\phi)f(\theta-\phi)\:\mathrm{d}\phi,
\end{equation*}
where $\vd[]{L}(\phi):=\vd[1]{L}(\phi):=\frac{\sin((L+1/2)\phi)}{\sin(\phi/2)}$ is the Dirichlet kernel of order $L$, and $\theta\in (-\pi,\pi]$.

For $0<\delta<\pi$, let $U(\theta;\delta):=\{\phi\in (-\pi,\pi]:\:\cos(\phi-\theta)>\cos\delta\}$ be a neighborhood of $\theta$ with angular radius $\delta>0$.
Let
$$v_{L}^\delta(\phi):=\vdl[1,\delta]{L}(\phi):=\vd[]{L}(\phi)\left(1-\chi_{U(0;\delta)}(\phi)\right),$$
where $\chi_{A}$
 is the indicator function for the set $A$.
The $L$th \emph{local convolution of $f\in \Lp{1}{1}$} is
\begin{equation*}
  V_{L}^{\delta}(f;\theta):=\Vdl[1,\delta]{L}(f;\theta)
  :=\frac{1}{2\pi}\int_{[-\pi,\pi]\backslash U(\theta;\delta)}\hspace{-0.5cm}\vd[]{L}(\theta-\phi)f(\phi)\:\mathrm{d}\phi
  =\frac{1}{2\pi}\int_{-\pi}^{\pi}v_{L}^\delta(\phi) f(\theta-\phi)\:\mathrm{d}\phi.
\end{equation*}
Thus the $L$th local convolution of $f$ at $\theta$ is precisely the partial sum at $\theta$ of the Fourier series of the modified function obtained by replacing the value of $f$ by zero in the open set $U(\theta;\delta)$.  The Riemann localisation principle on the circle can then be restated as an assertion that the  local convolution of an integrable function decays to zero as $L\rightarrow \infty$,
\begin{equation}\label{eq:RFoS.intro-circle-1}
  \lim_{L\rightarrow \infty} V_{L}^{\delta}(f;\theta)= 0 \quad \forall \theta \in (-\pi,\pi].
\end{equation}

The convergence to zero of \eqref{eq:RFoS.intro-circle-1} is a simple consequence of the Riemann-Lebesgue lemma. This can be seen by writing
\begin{equation}\label{eq:RFoS.intro-circle-2}
  V_{L}^{\delta}(f;\theta)=\frac{1}{2\pi}\int_{-\pi}^{\pi}\left(A_{\delta,\theta}(\phi)\cos(L\phi)+B_{\delta,\theta}(\phi)\sin(L\phi)\right)\:\mathrm{d}\phi,
\end{equation}
where  $A_{\delta,\theta}(\phi) := f(\theta-\phi)\chi_{[-\pi,\pi]\backslash U(0;\delta)}(\phi)$,
 $B_{\delta,\theta}(\phi) := f(\theta-\phi)\cot(\phi/2)\chi_{[-\pi,\pi]\backslash U(0;\delta)}(\phi)$.
Both terms in \eqref{eq:RFoS.intro-circle-2} approach zero as $L\rightarrow\infty$ since $A_{\delta,\theta}, B_{\delta,\theta}$ are in $\Lp{1}{1}$.

A more precise estimate than \eqref{eq:RFoS.intro-circle-1} was proved by Telyakovski\u{\i} \cite[Theorem~1, p.~184]{Telyakovskii2007}, as follows.
\begin{lemma}\label{lm:RFoS.intro-1} For $f\in \Lp{1}{1}$, let $a_{0}:=\frac{1}{\pi}\int_{-\pi}^{\pi}f(\phi)\:\mathrm{d}{\phi}$. Then, for $0<\delta<\pi$,
\begin{equation*}
  |V_{L}^{\delta}(f;\theta)|\le \frac{c}{\delta}\left(\frac{|a_{0}|}{L}+\omega(f,L^{-1})_{\Lp{1}{1}}\right),\quad \hbox{for all~} \theta\in(-\pi,\pi],
\end{equation*}
where $c$ is an absolute constant and
    $\omega(f,\eta)_{\Lp{1}{1}}:=\sup_{|\phi|\le \eta}\int_{-\pi}^{\pi}|f(z+\phi)-f(z)|\:\mathrm{d} z$
is the $\mathbb{L}_{1}$ modulus of continuity of $f$.
\end{lemma}
For $f\in \Lp{p}{1}$ with $1\le p\le\infty$, this gives
\begin{equation}\label{eq:RFoS.UB LocConv_Sph1-2}
  \normb{V_{L}^{\delta}(f)}{\Lp{p}{1}}\le \frac{c}{\delta}\left(\frac{\norm{f}{\Lp{p}{1}}}{L}+\omega(f,L^{-1})_{\Lp{1}{1}}\right).
\end{equation}
Since the modulus of continuity $\omega(f,L^{-1})_{\Lp{1}{1}}$ converges to zero as $L\rightarrow \infty$, the right-hand side of \eqref{eq:RFoS.UB LocConv_Sph1-2} converges to zero. As $\lim_{L\rightarrow\infty}\norm{V_{L}^{\delta}(f)}{\Lp{p}{1}}=0$ holds for each $f\in \Lp{p}{1}$, we say that the Fourier convolution (Fourier partial sum) $V_{L}$ has the \emph{Riemann localisation property} for $\Lp{p}{1}$.

Lemma~\ref{lm:RFoS.intro-1} was stated earlier by Hille and Klein \cite{HiKl1954}, but with a proof that was unfortunately incorrect.

\subsection{Fourier case}
In this paper, we generalise the concept of Riemann localisation and Lemma~\ref{lm:RFoS.intro-1} to the unit sphere $\mathbb{S}^{d}$ for $d\ge2$.
The normalised Legendre polynomial for $\sph{d}$ is
\begin{equation}\label{eq:disNsph.normalised.Gegenbauer}
  \NGegen{\ell}(t):=\Jcb[\frac{d-2}{2},\frac{d-2}{2}]{\ell}(t)/\Jcb[\frac{d-2}{2},\frac{d-2}{2}]{\ell}(1),
\end{equation}
where $\Jcb{\ell}$ is the Jacobi polynomial for $\alpha,\beta>-1$. The dimension of the space $\shSp{\ell}$ of spherical harmonics of exact degree $\ell$ is
\begin{equation}\label{eq:disNsph.dim.sph.harmon}
    Z(d,\ell):=(2\ell+d-1)\frac{\Gamma(\ell+d-1)}{\Gamma(d)\Gamma(\ell+1)}\asymp \ell^{d-1},
\end{equation}
where $a_{\ell}\asymp b_{\ell}$ means that there exists a constant $c>0$, independent of $\ell$, such that $c^{-1}\:a_{\ell}\le b_{\ell}\le c\:a_{\ell}$.

Let $\Lp{p}{d}$, $1\le p<\infty$ denote the $\mathbb{L}_{p}$-function space with respect to the normalised surface measure $\sigma_{d}$  on $\sph{d}$ and let $\Lp{\infty}{d}:=C(\mathbb{S}^{d})$ be the continuous function space on $\sph{d}$. In particular, $\Lp{2}{d}$ forms a Hilbert space with inner product
  $\InnerL{f,g}:=\int_{\sph{d}}f(\PT{x})g(\PT{x})\IntDiff{x}$, $f,g\in\Lp{2}{d}$.
For $f\in \Lp{1}{d}$, the projection onto $\shSp{\ell}$ of $f$ is
\begin{equation}\label{eq:RFoS.intro-preSph-1}
  \prj{\ell}(f;\PT{x}):=\InnerLb{f(\cdot),Z(d,\ell)\NGegen{\ell}(\PT{x}\cdot\cdot)}
  =\int_{\sph{d}}f(\PT{y})Z(d,\ell)\NGegen{\ell}(\PT{x}\cdot\PT{y})\IntDiff{y}.
\end{equation}
The Fourier convolution of order $L$ for $f\in \Lp{1}{d}$ (or the Fourier-Laplace series partial sum of order $L$ for $f$) is defined as the sum of the first $L+1$ projections $\prj{\ell}(f)$
\begin{equation*}
  \Vd{L}(f;\PT{x}):=\sum_{\ell=0}^{L}\prj{\ell}(f;\PT{x}),\quad \PT{x}\in\sph{d}.
\end{equation*}
By \eqref{eq:RFoS.intro-preSph-1},
\begin{equation*}
  \Vd{L}(f;\PT{x})=\InnerLb{f(\cdot),\vd{L}(\PT{x}\cdot\cdot)}=\int_{\sph{d}}\vd{L}(\PT{x}\cdot\PT{y})f(\PT{y})\IntDiff{y},
\end{equation*}
where $\vd{L}(\PT{x}\cdot\PT{y})$ is a zonal kernel (i.e. it depends only on $\PT{x}\cdot\PT{y}$) given by
\begin{equation}\label{eq:RFoS.intro.Dirichlet.kernel.sphere}
  \vd{L}(t):= \sum_{\ell=0}^{L} Z(d,\ell)\NGegen{\ell}(t),\quad t\in [-1,1].
\end{equation}

The metric on $\sph{d}$ may be defined by
  $\dist{\PT{x},\PT{y}}:=\arccos(\PT{x}\cdot \PT{y})$, $\PT{x},\PT{y}\in \sph{d}$,
the geodesic distance between $\PT{x}$ and $\PT{y}$.
Let
    $\scap{\PT{x},\delta}:=\left\{\PT{z}\in\sph{d}:\dist{\PT{x},\PT{z}}\le \delta\right\}$
be the spherical cap with center at $\PT{x}$ and geodesic radius $\delta$. By analogy with the case of the circle, we define the \emph{Fourier local convolution} of order $L$ with $f\in \Lp{1}{d}$ by
\begin{equation*}
  \Vdl{L}(f;\PT{x}):=\int_{\sph{d}\backslash \scap{\PT{x},\delta}}\vd{L}(\PT{x}\cdot \PT{y})f(\PT{y})\IntDiff{y},\quad \PT{x}\in \sph{d}.
\end{equation*}
In particular, when $\delta=0$, $\Vdl{L}$ reduces to the Fourier convolution $\Vd{L}$.

For $1\le p\le\infty$, we say the Fourier convolution $\Vd{L}$ has the \emph{Riemann localisation property} for a subset $X$ of $\mathbb{L}_{p}$ if there exists a $\delta_{0}>0$ such that for each $0<\delta<\delta_{0}$ the $\mathbb{L}_{p}$-norm of its local convolution $\Vdl{L}(f)$ decays to zero for all $f\in X$, i.e. if
\begin{equation*}
\lim_{L\rightarrow\infty}\|\Vdl{L}(f)\|_{\Lp{p}{d}}=0,\quad f\in X.
\end{equation*}
The behavior of the Fourier local convolution is characterised by the following theorems, which are proved as Theorem~\ref{thm:RFoS.upper.bound.local.conv.sphere}, Corollary~\ref{cor:RFoS.UpperBDSph-2} and Theorem~\ref{thm:RFoS.lower.bound.loc.Fourier.sph.const} respectively.
\begin{theorem*}[$\mathbb{L}_{p}$ upper bound for $\sph{d}$]\label{thm:RFoS.intro-Sphere-Lp}
Let $d$ be an integer and $p,\delta$ be real numbers satisfying $d\ge 2$, $1\le p\le\infty$ and $0<\delta<\pi$. For $f\in \Lp{p}{d}$ and positive integer $L$, there exists a constant $c$ depending only on $d$, $p$ and $\delta$ such that
\begin{equation}\label{eq:RFoS.intro-Sph-UpperBD-2}
  \normb{\Vdl{L}(f)}{\Lp{p}{d}}\le
  c\:L^{\frac{d-1}{2}}\left(L^{-1}\norm{f}{\Lp{p}{d}}+\:\modub{f,L^{-\frac{1}{2}}}\right),
\end{equation}
where $\modu{f,\cdot}$ is the $\Lp{p}{d}$-modulus of continuity of $f$, see \eqref{eq:RFoS.modulus.continuity} below.
\end{theorem*}

Let $\LBo$ be the Laplace-Beltrami operator on $\mathbb{S}^{d}$. Given $s>0$,
$\sob{p}{s}{d}:=\bigl\{g\in \Lp{p}{d}: (-\LBo)^{s/2} g\in \Lp{p}{d}\bigr\}$
is the Sobolev space of order $s$ on $\sph{d}$ with norm
    $\norm{f}{\sob{p}{s}{d}}:=\norm{f}{\Lp{p}{d}}+\norm{(-\LBo)^{s/2} f}{\Lp{p}{d}}$,
see e.g. \cite[Definition~4.3.3, p.~172]{WaLi2006}.
We have the following upper bound for a sufficiently smooth function $f$.
\begin{corollary*}[Upper bound for sufficiently smooth $f$]\label{cor:RFoS.intro-UpperBDSphere-Lp}
Let $d\ge 2$, $1\le p\le\infty$ and $0<\delta<\pi$. Then, for $f\in \sob{p}{s}{d}$, $s\ge2$, and $L\ge1$,
\begin{equation}\label{eq:RFoS.intro.sphere.upper.bound}
  \normb{\Vdl{L}(f)}{\Lp{p}{d}}\le
  c\:L^{\frac{d-3}{2}}\norm{f}{\sob{p}{s}{d}},
\end{equation}
where the constant $c$ depends only on $d,p,s$ and $\delta$.
\end{corollary*}

For $d=2$, the upper bound \eqref{eq:RFoS.intro.sphere.upper.bound} implies that the Fourier convolution $\Vd[2]{L}$ has the Riemann localisation property for $\sob{p}{s}{2}$ with $s\ge2$. However, \eqref{eq:RFoS.intro.sphere.upper.bound} gives no such assurance for $\sob{p}{s}{d}$ for $d\ge3$. The following lower bound tells us that in general the Riemann localisation property does \emph{not} hold for the Fourier convolution when $d\ge3$. Let $\mathbf{1}$ be the constant function on $\sph{d}$ satisfying $\mathbf{1}(\PT{x})=1$, $\PT{x}\in\sph{d}$.
\begin{theorem*}[A lower bound for $\sph{d}$]\label{thm:RFoS.intro-LowerBDSph} Let $d\ge2$, $1\le p\le \infty$ and $0<\delta<\pi/2$. Then there exists a subsequence $\{L_{\ell}\}_{\ell\ge1}\subset\Zp$ such that for $\ell\ge1$,
\begin{equation}\label{eq:RFoS.lower.bound.loc.Fourier.sph.const-intro}
  \normau{\Vdl{L_{\ell}}(\mathbf{1})}{\Lp{p}{d}} \ge c\: L_{\ell}^{\frac{d-3}{2}},
\end{equation}
where the positive constant $c$ depends only on $d$ and $\delta$.
\end{theorem*}
Since the constant function $\mathbf{1}$ is in every $\sob{p}{s}{d}$, $d\ge2$, $1\le p\le\infty$ and $s>0$, the lower bound in \eqref{eq:RFoS.lower.bound.loc.Fourier.sph.const-intro} shows that the Fourier convolution does not have the Riemann localisation property for $\sob{p}{s}{d}$ when $d\ge3$. Moreover, this lower bound implies that the upper bound of \eqref{eq:RFoS.intro.sphere.upper.bound} cannot be improved for $\sob{p}{s}{d}$ with $s\ge2$.

The upper bound \eqref{eq:RFoS.intro.sphere.upper.bound} with $d=2$ and $p=\infty$ shows that for $f\in \sob{\infty}{s}{2}$ with $s\ge2$, the Fourier partial sum $\Vd[2]{L}(f)$ converges pointwise to zero in any open subset on which $f$ vanishes.

Many authors have studied the localisation principle in a pointwise sense for general $d$. For Euclidean spaces and other manifolds including spheres, hyperbolic spaces and flat tori, see \cite{BoCl1973,BrCo1999, CaSo1988, CaSo1997Poi, Caso1997Set, Pinsky1994, Pinsky1995, PiTa1997, Taylor1999, Taylor2001, Taylor2002}. In particular, Bonami and Clerc~\cite{BoCl1973} showed that if $f \in \Lp{1}{1}$ vanishes in a neighborhood of $\theta\in(-\pi,\pi]$ then the Fourier partial sum $V^1_{L}(f;\theta)\to0$ as $L\to\infty$, but  that $V^d_{L}$ does not have an analogous  localisation property for $d \ge 2$. In this paper, we provide precise estimates for the Fourier local convolution on $\sph{d}$. This implies that the localisation principle for Fourier partial sums holds for $\sph{2}$ but not for higher dimensional spheres, as pointed out by Brandolini and Colzani, see \cite[p.~441--442]{BrCo1999}.

\subsection{Filtered case}
One way of improving the localisation of the Fourier-Laplace series partial sum is to modify the Fourier coefficients by the inclusion of an appropriate filter.
\begin{definition}\label{def:disNsph.fil.fil.ker}
A continuous compactly supported function $\fil: \mathbb{R}_{+}\to\mathbb{R}_{+}$ is said to be a filter. We will only consider filters with support a subinterval of $[0,2]$.

A filtered kernel on $\sph{d}$ with filter $\fil$ is, for $T\in \mathbb{R}_{+}$,
\begin{equation}\label{eq:disNsph.filter.sph.ker}
  \vdh{T,\fil}(\PT{x}\cdot\PT{y}):= \vdh[d]{T,\fil}(\PT{x}\cdot\PT{y}) :=\begin{cases}
  1, & 0\le T<1,\\[1mm]
  \displaystyle\sum_{\ell=0}^{\infty}\fil\Bigl(\frac{\ell}{T}\Bigr)\:Z(d,\ell)\:\NGegen{\ell}(\PT{x}\cdot\PT{y}), & T\ge1.
  \end{cases}
\end{equation}
\end{definition}

We may define a \emph{filtered (polynomial) approximation} $\Vdh{T,\fil}$ on $\Lp{1}{d}$, $T\ge0$ as an integral operator with the filtered kernel $\vdh{T,\fil}(\PT{x}\cdot\PT{y})$: for $f\in \Lp{1}{d}$,
\begin{equation}\label{eq:filter.sph.approx}
  \Vdh{T,\fil}(f;\PT{x}):= \Vdh[d]{T,\fil}(f;\PT{x}):= \InnerL{f,\vdh{T,\fil}(\PT{x}\cdot\cdot)}
  =\int_{\sph{d}} f(\PT{y})\:\vdh{T,\fil}(\PT{x}\cdot\PT{y})\:\IntDiff{y}.
\end{equation}
Note that for $T<1$ this is just the integral of $f$.

Let $\fil$ be a filter such that $\fil$ is constant on $[0,1]$ and $\supp \fil \subset [0,2]$ and let $\Vdh{L,\fil}$ be the filtered approximation defined by \eqref{eq:filter.sph.approx}. The \emph{filtered local convolution} $\Vdhl{L,\fil}$ for the filtered approximation $\Vdh{L,\fil}$ is defined by
\begin{equation*}
  \Vdhl{L,\fil}(f;\PT{x}):=\int_{\sph{d}\backslash \scap{\PT{x},\delta}}\vdh{L,\fil}(\PT{x}\cdot \PT{y})f(\PT{y})\IntDiff{y},\quad \PT{x}\in \sph{d}.
\end{equation*}
\begin{theorem*}[$\mathbb{L}_{p}$ upper bound for $\sph{d}$]\label{thm:intro filtered Sphere Lp}
Let $d\ge2$, $\kappa\in \Zp$, $1\le p\le \infty$, $0<\delta<\pi$. Let $\fil$ be a filter such that $\fil$ is constant on $[0,1]$ and $\supp \fil\subseteq [0,2]$ and\\
(i)~$\fil\in \CkR$;\\
(ii)~$\fil|_{[1,2]}\in \Ck{\kappa+3}{[1,2]}$.\\
Then, for $f\in \Lp{p}{d}$ and $L\in\Zp$,
\begin{equation*}
  \normb{\Vdhl{L,\fil}(f)}{\Lp{p}{d}}\le
  c\:L^{-(\kappa-\frac{d}{2}+\frac{3}{2})}\left(L^{-1}\norm{f}{\Lp{p}{d}}+\:\modub{f,L^{-\frac{1}{2}}}\right),
\end{equation*}
where the constant $c$ depends only on $d$, $p$, $\delta$ and $\fil$.
\end{theorem*}
For smoother functions, we have a simpler upper bound.
\begin{corollary*}[Upper bound for sufficiently smooth $f$]\label{cor:intro filtered Sphere Wp} Let $d\ge2$, $\kappa\in \Zp$, $1\le p\le \infty$, $0<\delta<\pi$. Let $\fil$ be a filter such that $\fil$ is constant on $[0,1]$ and $\supp \fil\subseteq [0,2]$ and\\
(i)~$\fil\in \CkR$;\\
(ii)~$\fil|_{[1,2]}\in \Ck{\kappa+3}{[1,2]}$.\\
Then, for $f\in \sob{p}{s}{d}$, $s\ge2$, and $L\in\Zp$,
\begin{equation*}
  \normb{\Vdhl{L,\fil}(f)}{\Lp{p}{d}}\le
  c\:L^{-(\kappa-\frac{d}{2}+\frac{5}{2})}\norm{f}{\sob{p}{s}{d}},
\end{equation*}
where the constant $c$ depends only on $d$, $p$, $s$, $\delta$ and $\fil$.
\end{corollary*}

We see from this corollary that the filtered local convolution of $f\in\sob{p}{s}{d}$, $s\ge2$, converges to zero for a sphere of arbitrary dimension if the filter function is sufficiently smooth. This improves the upper bound of the Fourier local convolution and thus improves the Riemann localisation of the Fourier convolution (the Fourier partial sum).

Localisation properties are critical in multiresolution analysis on the sphere. Many authors have investigated localisation from a variety of aspects, see e.g. \cite{AnVa2007, BaKeMaPi2009-1, FrMa2003, FrWi1997, Mhaskar2005, NaPeWa2006-2, SiDaWi2006,WaLeSlWo2016}.
The Riemann localisation property of the Fourier-Laplace series partial sum for $\sob{p}{s}{2}$ implies that the multiscale approximation converges to the solution of the local downward continuation problem, see \cite{FrMa2003, Gerhards2014}. The estimation of the local convolution also plays a role in the ``missing observation'' problem, see \cite[Section~10.5]{MaPe2011} and \cite{BaKeMaPi2009-2}.

The paper is organised as follows. Section~\ref{sec:RFoS.AsympJacobiDirichlet} contains the estimates of the generalised Dirichlet kernel and the filtered kernel for Jacobi weights. In Section~\ref{sec:RFoS.RLocSph} we use the results of Section~\ref{sec:RFoS.AsympJacobiDirichlet} to prove the upper and lower bounds for the Fourier local convolution for $\mathbb{L}_{p}$ spaces and Sobolev spaces on $\sph{d}$. In Section~\ref{sec:Rieloc.fisph} we prove an upper bound of the filtered local convolution for functions in $\mathbb{L}_{p}$ spaces and Sobolev spaces on $\sph{d}$. Section~\ref{sec:RFoS.some proofs} gives the proofs of results in Section~\ref{sec:RFoS.AsympJacobiDirichlet}.\\[-3mm]

\noindent\textbf{Notation.}
Let $\Rplus:=[0,+\infty)$ and $\Zp$ be the set of all positive integers and let $\Nz:=\Zp\cup \{0\}$.
Given $k\in \Nz$ and an interval $I$, either open, closed or half-open, let $\Ck{k}{I}$ be the space of $k$ times continuously differentiable functions on $I$. We let $\Ck{k}{a,b}:=\Ck{k}{(a,b)}$ for an open interval $(a,b)$.
For $f\in\Ck{k}{[a,b]}$, $k=0,1,\dots$, the left and right limits denoted by
  $f^{(k)}(a+):=\lim_{t\to a+}f^{(k)}(t)$, $f^{(k)}(b-):=\lim_{t\to b-}f^{(\fis+1)}(t)$
are assumed to exist.
For a function $g$ from a metric space $X$ to $\Rone$, let $\supp g$ be the support of $g$, the closure of the set of points where $g$ is non-zero:
  $\supp{g} := \overline{\{x\in X : g(x)\neq0\}}$.

Let $a(T), b(T)$ be two sequences (when $T\in \Zp$) or functions (when $T\in\mathbb{R}_{+}$) of $T$. The notation $a(T)\asymp_{\alpha} b(T)$ means that there is a real constant $c_{\alpha}>0$ depending only on $\alpha$ such that $c_{\alpha}^{-1}\:b(T)\le a(T)\le c_{\alpha}\: b(T)$; we write $a(T)\asymp b(T)$ if no confusion arises.
The big $\mathcal{O}$ notation $a(T)=\bigo{\alpha}{b(T)}$ means there exists a constant $c_{\alpha}>0$ and $T_{0}\in \mathbb{R}_{+}$ depending only on $\alpha$ such that $|a(T)|\le c_{\alpha}|b(T)|$ for all $T\ge T_{0}$.

We will use the asymptotic expansion of the Gamma function, as follows. Given $a,b\in\mathbb{R}$, see \cite[Eq.~5.11.13, Eq.~5.11.15]{NIST:DLMF},
\begin{equation}\label{eq:fiJ.Gamma.asymp.one.term}
  \frac{\Gamma(L+a)}{\Gamma(L+b)} = L^{a-b} + \bigo{a,b}{L^{a-b-1}}.
\end{equation}
The ceiling function $\left\lceil x\right\rceil$ is the smallest integer at least $x$ and the floor function $\left\lfloor x\right\rfloor$ is the largest integer at most $x$. For integer $k\ge0$ and real $a\ge k$, let
\begin{equation*}
    {a\choose k}:=\frac{a(a-1)\cdots (a-k+1)}{k!}=\frac{\Gamma(a+1)}{\Gamma(a-k+1)\Gamma(k+1)}
\end{equation*}
be the extended binomial coefficient. We use ``$L$'' as a non-negative integer and ``$T$'' as a positive real number.
We define $\ellshift:=\ellshift(\alpha,\beta):=\ell+\frac{\alpha+\beta+1}{2}$ as the shift of $\ell$, and
$\Lhat:=L+\frac{\alpha+\beta+1}{2}$ and
$\Lshift:=L+\frac{\alpha+\beta+2}{2}$ as the shifts of $L$.

\section{Asymptotic properties of kernels}\label{sec:RFoS.AsympJacobiDirichlet}
Characterisation of the Riemann localisation property on the sphere relies on the asymptotic estimate of the Dirichlet kernel $\vd{L}(t)$ and the filtered Jacobi kernel $\vdh{L,\fil}$ in Sections~\ref{subsec:RFoS.gasymp.Dirichlet.kernel} and \ref{subsec:RFoS.fi.ker.asymp} respectively.

The Jacobi weight function $\Jw(t)$ is
    $\Jw(t):=(1-t)^{\alpha}(1+t)^{\beta}$, $-1\le t\le1$,
where $\alpha,\beta>-1$ are fixed parameters.
The corresponding Jacobi polynomials $\Jcb{\ell}(t)$, $\ell=0,1,\dots$ form a complete orthogonal basis for the space $\Lpw{2}=\mathbb{L}_{2}([-1,1],\Jw)$, which is the $\mathbb{L}_{2}$ space on $[-1,1]$ with respect to the weight function $\Jw$.

We will use the value of $\Jcb{\ell}(1)$, see \cite[Eq.~4.1.1, p.~58]{Szego1975} or \cite[18.6.1]{NIST:DLMF}: given $\alpha,\beta>-1$,
\begin{equation}\label{eq:prefiJcb.Jcb.1}
    \Jcb{\ell}(1)={\ell+\alpha\choose \ell}=\frac{\Gamma(\ell+\alpha+1)}{\Gamma(\ell+1)\Gamma(\alpha+1)}.
\end{equation}

Adopting the normalisation of \cite[Eq.~4.3.3, p.~68]{Szego1975}, we have
\begin{equation*}
  \int_{-1}^{1}\Jcb{\ell}(t)\Jcb{\ell'}(t)\:\Jw(t)\IntD{t}= \delta_{\ell,\ell'}\:\Jcoe{\ell},
\end{equation*}
where $\delta_{\ell,\ell'}$ is the Kronecker delta and
\begin{equation}\label{eq:fiJ.Jacobi.normalisation.coe}
  \Jcoe{\ell}:=\frac{2^{\alpha+\beta+1}}{2\ell+\alpha+\beta+1}\frac{\Gamma(\ell+\alpha+1)\Gamma(\ell+\beta+1)}{\Gamma(\ell+1)\Gamma(\ell+\alpha+\beta+1)}.
\end{equation}

The $L$th partial sum of the Fourier series for $f\in \Lpw{1}$ is given by
\begin{equation*}
  \Vab{L}(f;t) = \sum_{\ell=0}^{L}\widehat{f}(\ell)\left(\Jcoe{\ell}\right)^{-\frac{1}{2}}\Jcb{\ell}(t),
 \end{equation*}
where $\widehat{f}(\ell)$ is the $\ell$th Fourier coefficient given by
  $\widehat{f}(\ell):=\InnerLab{f,\left(\Jcoe{\ell}\right)^{-\frac{1}{2}}\Jcb{\ell}}$.
Thus the Fourier partial sum can be written as
\begin{equation*}
 \Vab{L}(f;t) = \InnerLab{f(\cdot), \vab{L}(t,\cdot)},
\end{equation*}
in which $\vab{L}(t,s)$ is the (generalised) Dirichlet kernel (the ``Fourier'' kernel)
\begin{equation}\label{eq:fiJ.Dirichlet.Jacobi.kernel}
  \vab{L}(t,s):=\sum_{\ell=0}^{L}\left(\Jcoe{\ell}\right)^{-1}\Jcb{\ell}(t) \Jcb{\ell}(s).
\end{equation}

The filtered approximation with a filter $\fil$ and $\supp\fil\subseteq[0,2]$ for the Jacobi weight $\Jw$ is the polynomial of degree at most $2L-1$ defined by
\begin{align*}
  \Vabh{L,\fil}(f;t)
  &:= \sum_{\ell=0}^{\infty}\fil\biggl(\frac{\ell}{L}\biggr)\:\widehat{f}(\ell)
  \left(\Jcoe{\ell}\right)^{-\frac{1}{2}}\Jcb{\ell}(t)\notag\\
  &= \sum_{\ell=0}^{2L-1}\fil\biggl(\frac{\ell}{L}\biggr)\:\widehat{f}(\ell)
  \left(\Jcoe{\ell}\right)^{-\frac{1}{2}}\Jcb{\ell}(t)
  = \InnerLab{f(\cdot), \vabh{L,\fil}(t,\cdot)},
\end{align*}
where the filtered kernel $\vab{L,\fil}(t,s)$ takes the form \cite[(1.2), p.~558]{PeXu2005}
\begin{equation}\label{eq:fiJ.filtered.Jacobi.kernel}
  \vab{L,\fil}(t,s)
  =\sum_{\ell=0}^{2L-1}\fil\left(\frac{\ell}{L}\right) \left(\Jcoe{\ell}\right)^{-1}
  \Jcb{\ell}(t)\Jcb{\ell}(s).
\end{equation}

Now we return to the setting of the sphere $\sph{d}$. The area of $\sph{d}$, $d\ge1$, is
\begin{equation}\label{eq:presphhar.area.sph}
    |\sph{d}| =  \frac{2\pi^{\frac{d+1}{2}}}{\Gamma(\frac{d+1}{2})}.
\end{equation}
The Fourier convolution kernel $\vd{L}(t)$, $t\in[-1,1]$, in \eqref{eq:RFoS.intro.Dirichlet.kernel.sphere} is a constant multiple of $\vab{L}(1,t)$ with $\alpha=\beta=(d-2)/2$ in \eqref{eq:fiJ.Dirichlet.Jacobi.kernel}:
\begin{lemma}\label{lm:prefisph.vd.vab} Let $d\ge2$ and $L\ge0$. Then, for $t\in[-1,1]$,
\begin{equation}\label{eq:RFoS.vd.vab}
  \vd{L}(t)=\sqrt{\pi}\frac{\Gamma(\frac{d}{2})}{\Gamma(\frac{d+1}{2})}\:\vab[\frac{d-2}{2},\frac{d-2}{2}]{L}(1,t)=\frac{|\sph{d}|}{|\sph{d-1}|}\:\vab[\frac{d-2}{2},\frac{d-2}{2}]{L}(1,t).
\end{equation}
\end{lemma}

We give the proof of Lemma~\ref{lm:prefisph.vd.vab} in Section~\ref{sec:RFoS.some proofs}.

Using \eqref{eq:disNsph.filter.sph.ker} with \eqref{eq:disNsph.normalised.Gegenbauer}, \eqref{eq:disNsph.dim.sph.harmon} and \eqref{eq:prefiJcb.Jcb.1} gives for $T\ge1$
\begin{align*}
    \vdh{T,\fil}(t)&= \sum_{\ell=0}^{\infty} \fil\left(\frac{\ell}{T}\right) Z(d,\ell) \NGegen{\ell}(t)\\
    &= \frac{\Gamma(\frac{d}{2})}{\Gamma{(d)}}\sum_{\ell=0}^{\infty} \fil\left(\frac{\ell}{T}\right)
    \frac{(2\ell+d-1)\Gamma(\ell+d-1)}{\Gamma(\ell+\frac{d}{2})}\Jcb[\frac{d-2}{2},\frac{d-2}{2}]{\ell}(t).
\end{align*}
The following lemma shows that it is a constant multiple of the filtered Jacobi kernel in \eqref{eq:fiJ.filtered.Jacobi.kernel}, cf. Lemma~\ref{lm:prefisph.vd.vab}.
\begin{lemma}\label{lm:RFoS.vdh.vabh} Let $d\ge2$ and $L\in \Zp$. Then, for $t\in[-1,1]$,
\begin{equation*}
    \vdh{L,\fil}(t) = \sqrt{\pi}\frac{\Gamma(\frac{d}{2})}{\Gamma(\frac{d+1}{2})}\:\vabh[(\frac{d-2}{2},\frac{d-2}{2})]{L,\fil}(1,t)
                    = \frac{|\sph{d}|}{|\sph{d-1}|}\vabh[(\frac{d-2}{2},\frac{d-2}{2})]{L,\fil}(1,t).
\end{equation*}
\end{lemma}
The proof of Lemma~\ref{lm:RFoS.vdh.vabh} is similar to that of Lemma~\ref{lm:prefisph.vd.vab}.

\subsection{Asymptotic expansions for Jacobi polynomials}\label{subsec:RFoS.asymp.Jacobi}
Our estimate is based on the following asymptotic expansion for Jacobi polynomials.
\begin{lemma}\label{lm:RFoS.Jacobi.asymp} i)~Given $\alpha,\beta$ such that $\alpha>-1$, $\beta>-1$, there exists a constant $c>0$ depending on $\alpha,\beta$ such that for $c\: \ell^{-1} \le\theta\le\pi-c\:\ell^{-1}$, $\ell\ge1$,
\begin{equation}
    \Jcb{\ell}(\cos\theta)
    =\ellshift^{-\frac{1}{2}}\: m_{\alpha,\beta}(\theta)
    \left(\cos\omega_{\alpha}(\ellshift\theta)+
    (\sin\theta)^{-1}\bigo{\alpha,\beta}{\ell^{-1}}\right),\label{eq:RFoS.Jacobi.asymp-1}
\end{equation}
where
\begin{subequations}\label{subeqs:RFoS.Jacobi asymp-1}
\begin{align}
&\ellshift:=\ellshift(\alpha,\beta):=\ell+(\alpha+\beta+1)/2, \label{eq:RFoS.ell}\\
&m_{\alpha,\beta}(\theta):=\pi^{-\frac{1}{2}}
    \left(\sin\frac{\theta}{2}\right)^{-\alpha-\frac{1}{2}}\left(\cos\frac{\theta}{2}\right)^{-\beta-\frac{1}{2}},\label{eq:RFoS.m}\\
&\omega_{\alpha}(z):=z-\frac{\alpha\pi}{2}-\frac{\pi}{4}.\label{eq:RFoS.omega}
\end{align}
\end{subequations}

ii)~Let $\alpha,\beta>-1/2$, $\alpha-\beta>-4$ and $c\:\ell^{-1}\le\theta\le\pi-\epsilon$ with $\epsilon>0$. Then
\begin{align}\label{eq:RFoS.Jacobi.asymp.two.term}
    &\Jcb{\ell}(\cos\theta)= \ellshift^{-\frac{1}{2}}\:m_{\alpha,\beta}(\theta)\\
    &\hspace{7mm}\times\left[\cos\omega_{\alpha}(\ellshift\theta)+\:\ellshift^{-1}\:\FA{1}{\theta}
    +\bigo{\epsilon,\alpha,\beta}{\ell^{\:\widehat{u}(\alpha)}\theta^{\widehat{\nu}(\alpha)}}
    +\bigo{\alpha,\beta}{\ell^{-2}\theta^{-2}}\right],\nonumber
\end{align}
where
\begin{subequations}\label{subeqs:RFoS.FA.FB}
\begin{align}
    &\FA[\ellshift]{1}{\theta}:=\FB{\theta} \cos\omega_{\alpha+1}(\ellshift\theta)-\frac{\alpha\beta}{2}\cos\omega_{\alpha}(\ellshift\theta),\label{eq:RFoS.F^1}\\[0.2cm]
    &\FB{\theta} :=\frac{\beta^{2}-\alpha^{2}}{4}\tan\frac{\theta}{2}-\frac{4\alpha^{2}-1}{8}\cot\theta,\label{eq:RFoS.F^2}\\[0.2cm]
    &\widehat{u}(\alpha) := -2+\fractional{\alpha+\tfrac{1}{2}},\quad
    \widehat{\nu}(\alpha) :=\left\{\begin{array}{ll}
        \alpha+\frac{5}{2},&\alpha<\frac{1}{2},\\
        \alpha+\frac{1}{2},&\alpha\ge\frac{1}{2},
    \end{array}\right.\label{eq:RFoS.hat.u.hat.nu}
\end{align}
\end{subequations}
where $\fractional{x}:=x-\floor{x}$ denotes the fractional part of a real number $x$.
\end{lemma}

\begin{remark} For $\alpha\ge1/2$, the condition ``$\alpha-\beta>-4$'' may be weakened to ``$\alpha-\beta>-4-2\floor{\frac{1}{2}+\alpha}$'', see the proof of Lemma~\ref{lm:RFoS.Jacobi.asymp}. Also, we observe that $\widehat{u}(\alpha)<-1$ and $\widehat{\nu}(\alpha)\ge1$.
\end{remark}

Lemma~\ref{lm:RFoS.Jacobi.asymp} ii) is a corollary of Frenzen and Wong's expansion of the Jacobi polynomial in terms of the Bessel functions, see \cite[Main Theorem, p.~980]{FrWo1985}. The jump of $\widehat{\nu}(\alpha)$ at $\alpha=1/2$ in \eqref{eq:RFoS.hat.u.hat.nu} is due to the jump of the power of $\theta$ in the remainder of the expansion. See the proof of Lemma~\ref{lm:RFoS.Jacobi.asymp} in Section~\ref{subsec:RFoS.proof.asymp.Jacobi} for details.

\subsection{Asymptotic estimates for Dirichlet kernels}\label{subsec:RFoS.gasymp.Dirichlet.kernel}
With the help of Lemma~\ref{lm:RFoS.Jacobi.asymp}, we may prove Lemmas~\ref{lm:RFoS.Dirichlet.Jacobi.one.term.asymp} and \ref{lm:RFoS.Dirichlet.Jacobi.two.term.asymp} below, which show how the generalised Dirichlet kernel $\vab{L}(1,s)$ behaves as $L\rightarrow+\infty$. We prove both one-term and two-term asymptotic expansions of the generalised Dirichlet kernel $\vab{L}(1,s)$. The one-term expansions are utilised to prove the upper bounds on the Fourier local convolution, while the two-term expansion plays an important role in the estimate of the lower bound. Adopting the notation of \eqref{subeqs:RFoS.Jacobi asymp-1} and \eqref{subeqs:RFoS.FA.FB}, we have
\begin{lemma}\label{lm:RFoS.Dirichlet.Jacobi.one.term.asymp} Let $\alpha>-1/2$, $\beta>-1/2$ and $0<\theta<\pi$. For $L\in \Zp$, let
\begin{equation*}
  \Lshift:= L+(\alpha+\beta+2)/2.
\end{equation*}
Then there exists a constant $c^{(1)}$ depending only on $\alpha,\beta$ such that:\\
\begin{subequations}\label{eq:RFoS.Dirichlet.Jacobi.one.term.asymp}
i)~For $c^{(1)}L^{-1}\le\theta\le\pi/2$,
\begin{equation}\label{eq:RFoS.Dirichlet.Jacobi.one.term.asymp-a}
\vab{L}(1,\cos\theta)
=\frac{2^{-(\alpha+\beta+1)}}{\Gamma(\alpha+1)}\:\Lshift^{\alpha+\frac{1}{2}}\:m_{\alpha+1,\beta}(\theta)
\left(\cos\omega_{\alpha+1}(\Lshift\:\theta)+
  (\sin\theta)^{-1}\bigo{\alpha,\beta}{L^{-1}}\right).
\end{equation}
ii)~For $\pi/2<\theta\le\pi-c^{(1)}L^{-1}$, letting $\theta':=\pi-\theta$,
\begin{equation}\label{eq:RFoS.Dirichlet.Jacobi.one.term.asymp-b}
\vab{L}(1,\cos\theta)
=\frac{2^{-(\alpha+\beta+1)}}{\Gamma(\alpha+1)} \Lshift^{\alpha+\frac{1}{2}}(-1)^{L} m_{\beta,\alpha+1}(\theta')\:\left(\cos\omega_{\beta}(\Lshift\:\theta')+
  (\sin\theta')^{-1}\bigo{\alpha,\beta}{L^{-1}}\right),
\end{equation}
where the constants in the error terms of \eqref{eq:RFoS.Dirichlet.Jacobi.one.term.asymp-a} and \eqref{eq:RFoS.Dirichlet.Jacobi.one.term.asymp-b} depend only on $\alpha,\beta$.
\end{subequations}
\end{lemma}

\begin{lemma}\label{lm:RFoS.Dirichlet.Jacobi.two.term.asymp}
i) Let $\alpha,\beta>-1/2$ satisfying $\alpha-\beta>-5$, and $0<\epsilon<\pi/2$. Then,
for $c^{(1)} L^{-1}\le\theta\le\pi-\epsilon$,
\begin{align*}
  &\hspace{-0.2cm}\vab{L}(1,\cos\theta)
  =\frac{2^{-(\alpha+\beta+1)}}{\Gamma(\alpha+1)}\:\Lshift^{\alpha+\frac{1}{2}}\:m_{\alpha+1,\beta}(\theta)\\
  &\hspace{0.4cm}\times\Bigl[\cos\omega_{\alpha+1}(\Lshift\:\theta)+\Lshift^{-1}\FA[\Lshift]{3}{\theta}
  +\bigo{\epsilon,\alpha,\beta}{L^{\hspace{0.2mm}\widehat{u}(\alpha+1)}\theta^{\widehat{\nu}(\alpha+1)}}
  +\bigo{\alpha,\beta}{L^{-2}\theta^{-2}}\Bigr],\nonumber
\end{align*}
where
\begin{equation*}
    \FA[\Lshift]{3}{\theta} := \FB[\alpha+1,\beta]{\theta}\cos\omega_{\alpha+2}(\Lshift\theta),
\end{equation*}
and $\FB[\alpha+1,\beta]{\theta}$ is given by \eqref{eq:RFoS.F^2}.

ii) Let $\alpha,\beta>-1/2$ satisfying $\beta-\alpha>-3$ and let $\epsilon\le\theta<\pi-c^{(1)}L^{-1}$ with $0<\epsilon<\pi/2$, and $\theta':=\pi-\theta$.
Then
\begin{align}\label{eq:RFoS.Dirichlet.Jacobi.two.term.asymp-b}
  &\hspace{-0.5cm}\vab{L}(1,\cos\theta)
  =\frac{(-1)^{L}2^{-(\alpha+\beta+1)}}{\Gamma(\alpha+1)}\:\Lshift^{\alpha+\frac{1}{2}}\:m_{\beta,\alpha+1}(\theta')\\
  &\hspace{0.4cm}\times\Bigl[\cos\omega_{\beta}(\Lshift\:\theta')+\Lshift^{-1}\FA[\Lshift]{4}{\theta'}
  +\bigo{\epsilon,\alpha,\beta}{L^{\hspace{0.2mm}\widehat{u}(\beta)}\theta'^{\hspace{0.2mm}\widehat{\nu}(\beta)}}
  +\bigo{\alpha,\beta}{L^{-2}\theta'^{-2}}\Bigr],\nonumber
\end{align}
where
\begin{equation}\label{eq:RFoS.FA4}
    \FA[\Lshift]{4}{\theta'} := \FB[\beta,\alpha+1]{\theta'}\cos\omega_{\beta+1}(\Lshift\theta').
\end{equation}
\end{lemma}
The proofs of Lemmas~\ref{lm:RFoS.Dirichlet.Jacobi.one.term.asymp} and \ref{lm:RFoS.Dirichlet.Jacobi.two.term.asymp} are given in Section~\ref{subsec:RFoS.gproof.asymp.Dirichlet.kernel}.

Note that Lemmas~\ref{lm:RFoS.Dirichlet.Jacobi.one.term.asymp} and \ref{lm:RFoS.Dirichlet.Jacobi.two.term.asymp} do not describe the behavior of $v_{L}^{(\alpha,\beta)}(1,\cos\theta)$ near the two ends of the interval $[0,\pi]$. This is given by the following lemma. The proof is again given in Section~\ref{subsec:RFoS.gproof.asymp.Dirichlet.kernel}.
\begin{lemma}\label{lm:RFoS.Dirichlet.Jacobi.kernel.upper.bound} For $\alpha,\beta>-1/2$, adopting the notation of Lemma~\ref{lm:RFoS.Dirichlet.Jacobi.one.term.asymp},\\
\begin{subequations}\label{eq:RFoS.Dirichlet.Kernel.Jacobi.Upper}
i)~$\hbox{for}~0\le\theta\le c^{(1)}L^{-1}$,
\begin{equation}\label{eq:RFoS.Dirichlet.Jacobi.upper.bound-a}
\vab{L}(1,\cos\theta)
=\mathcal{O}_{\alpha,\beta}(L^{2\alpha+2}),
\end{equation}
ii)~$\hbox{for}~\pi-c^{(1)}L^{-1}\le\theta\le\pi$,
\begin{equation}\label{eq:RFoS.Dirichlet.Jacobi.upper.bound-b}
\vab{L}(1,\cos\theta)
=\mathcal{O}_{\alpha,\beta}(L^{\alpha+\beta+1}).
\end{equation}
\end{subequations}
\end{lemma}

\subsection{Asymptotic estimates for filtered Jacobi kernel}\label{subsec:RFoS.fi.ker.asymp}
The following theorem shows an asymptotic expansion of $\vab{L,\fil}(1,\cos\theta)$. We will exploit this result to prove the upper bound of the filtered local convolution on $\sph{d}$.

Given $s\in\Zp$, $\nu\in\Nz$ satisfying $0\le\nu\le s-1$, let
\begin{subequations}\label{subeqs:fiJ.lambda.nu.s}
\begin{equation}\label{eq:lambda.nu.s-a}
\lambda_{\nu,s}^{\fis} := \sum_{j=\nu+1}^{s}{s\choose j}(-1)^{j}(j-\nu)^{\fis+1},
\end{equation}
and for $0\le\nu\le s$, let
\begin{equation}\label{eq:lambda.nu.s-b}
 \overline{\lambda}_{\nu,s}^{\fis}:=\sum_{j=0}^{\nu}{s\choose j}(-1)^{j}(j-\nu-1)^{\fis+1}.
\end{equation}
\end{subequations}

\begin{theorem}[Asymptotic expansion of filtered kernel]\label{thm:fiJ.filtered.kernel.asymp-2} Let $\alpha,\beta>-1$, $\fis\in \Zp$. Let $\fil$ be a filter such that $\fil(t)=c$ for $t\in [0,1]$ with $c\ge0$ and $\supp \fil\subseteq [0,2]$ and\\
(i)~~$\fil\in \CkR$;\\
(ii)~~$\fil|_{[1,2]}\in\Ck{\fis+1}{[1,2]}$;\\
(iii)~~$\fil|_{(1,2)}\in \Ck{\fis+3}{1,2}$;\\
(iv)~~$\fil^{(i)}|_{(1,2)}$ is bounded on $(1,2)$, $i=\fis+2,\fis+3$.\\
Then for $c\:L^{-1}\le\theta\le \pi-c\:L^{-1}$ with some $c>0$,
\begin{align}\label{eq:vabh.asymp.expan}
  \vabh{L,\fil}(1,\cos\theta)\nonumber
   & = L^{-\left(\fis-\alpha+\frac{1}{2}\right)}\:\frac{C^{(1)}_{\alpha,\beta,\fis+3}(\theta)}{2^{\fis+3}(\fis+1)!}
    \bigl(\afiJcbu{1}(\theta)\cos\phi_{L}(\theta)+ \afiJcbu{2}(\theta)\sin\phi_{L}(\theta)\\
  &\qquad +\afiJcbu{3}(\theta)\cos\overline{\phi}_{L}(\theta)+
\afiJcbu{4}(\theta)\sin\overline{\phi}_{L}(\theta)+(\sin\theta)^{-1}\:\bigo{\alpha,\beta,\fil,\fis}{L^{-1}}\bigr),
\end{align}
where
\begin{equation*}
\begin{array}{l}
    \displaystyle C^{(1)}_{\alpha,\beta,k}(\theta) := \frac{\left(\sin\frac{\theta}{2}\right)^{-\alpha-k-\frac{1}{2}}\left(\cos\frac{\theta}{2}\right)^{-\beta-\frac{1}{2}}}{2^{\alpha+\beta+1}\sqrt{\pi}\:\Gamma(\alpha+1)}\\[0.33cm]
    \displaystyle \afiJcbu{1}(\theta) := \fil^{(\fis+1)}(1+)\sum_{i=0}^{\fis+2}\lambda_{i,\fis+3}^{\fis}\cos(i\theta),\quad \afiJcbu{3}(\theta) := 2^{\alpha+\frac{1}{2}}\fil^{(\fis+1)}(2-)\sum_{i=0}^{\fis+2}\overline{\lambda}_{i,\fis+3}^{\fis}\cos(i\theta),\\
    \displaystyle \afiJcbu{2}(\theta) := \fil^{(\fis+1)}(1+)\sum_{i=0}^{\fis+2}\lambda_{i,\fis+3}^{\fis}\sin(i\theta),\quad \afiJcbu{4}(\theta) := 2^{\alpha+\frac{1}{2}}\fil^{(\fis+1)}(2-)\sum_{i=0}^{\fis+2}\overline{\lambda}_{i,\fis+3}^{\fis}\sin(i\theta),
  \end{array}
\end{equation*}
where $\lambda_{i,\fis+3}^{\fis}$ and $\overline{\lambda}_{i,\fis+3}^{\fis}$ are given by \eqref{subeqs:fiJ.lambda.nu.s}, and $\afiJcbu{1}(\theta)$ can be written as an algebraic polynomial of $\cos\theta$ of degree $\fis+2$ with initial coefficient $(-2)^{\fis+1}\fil^{(\fis+1)}(1+)$, and
\begin{equation*}
\phi_{L}(\theta) :=\bigl(\Lshift+\tfrac{\fis+2}{2}\bigr)\theta-\xi_{1},
\quad \overline{\phi}_{L}(\theta) :=\bigl(\widetilde{2L}-1+\tfrac{\fis+2}{2}\bigr)\theta-\xi_{1},
\end{equation*}
where $\Lshift := L + \tfrac{\alpha+\beta+2}{2}$, $\Lshift[2L]:=2L + \tfrac{\alpha+\beta+2}{2}$ and $\xi_{1}:=\tfrac{\alpha+\fis+3}{2}\pi+\tfrac{\pi}{4}$.
\end{theorem}

The proof of Theorem~\ref{thm:fiJ.filtered.kernel.asymp-2} is given in Section~\ref{subsec:proof.fi.ker.asymp}.

\section{Fourier local convolution on the sphere}\label{sec:RFoS.RLocSph}
We focus in this section on the proofs of the main theorems for the Fourier case. The upper bound \eqref{eq:RFoS.intro-Sph-UpperBD-2} and the lower bound \eqref{eq:RFoS.lower.bound.loc.Fourier.sph.const-intro} are proved in Theorem~\ref{thm:RFoS.upper.bound.local.conv.sphere} and in Theorem~\ref{thm:RFoS.lower.bound.loc.Fourier.sph.const} respectively. The upper bound of the theorem comes from the asymptotic behavior of the generalised Dirichlet kernel (the one-term expansions, see Lemma~\ref{lm:RFoS.Dirichlet.Jacobi.one.term.asymp}).

For $\PT{x}\in\sph{d}$, let $\Tx:=\{\PT{y}\in\sph{d} |\PT{x}\cdot\PT{y}=\cos\theta\}$ be the boundary of the spherical cap on $\sph{d}$ with center at $\PT{x}$. The set $\Tx$ can be regarded as a $(d-1)$-dimensional subset of $\sph{d}$. We shall make repeated use of $\Tsph{\theta}(f;\PT{x})$, the \emph{translation operator} for $f\in \Lp{1}{d}$, given by, see e.g. \cite[Section~2.4, p. 57]{WaLi2006},
\begin{equation*}
  \Tsph{\theta}(f;\PT{x}):=\Tsph[(d)]{\theta}(f;\PT{x}):=\frac{1}{|\mathbb{S}^{d-1}|(\sin\theta)^{d-1}}\int_{\Tx}f(\PT{y})\: \IntDiff[d-1]{y},\quad 0<\theta\le \pi,
\end{equation*}
where ${\sigma}_{d-1}$ is the surface measure on $\sph{d-1}$. We also write $\Tsph{0}(f;\PT{x}):=\Tsph[(d)]{0}(f;\PT{x}):=f(\PT{x})$. Thus the translation of $\PT{x}$ is just the average of $f$ over arcs of constant latitude with respect to $\PT{x}$ as a pole.
From the following formula we can see that $\Tsph{\theta}(f;\PT{x})$ is a continuous function of $\theta$:
\begin{equation*}
    \Tsph{\theta}(f;\PT{x}) = \int_{\Tx}f(\PT{x}\cos\theta+\xi\cos\theta)\IntDiff[d-1]{\xi}, \quad 0<\theta\le \pi,\; \PT{x}\in\sph{d}.
\end{equation*}
Note that for any zonal kernel $v\in \mathbb{L}_{1}\left([-1,1],\Jw[\frac{d-2}{2},\frac{d-2}{2}]\right)$ we can write
\begin{equation}\label{eq:RFoS.integral.zonal.via.translation}
  \int_{\sph{d}}v(\PT{x}\cdot\PT{y})f(\PT{y})\IntDiff{y}=\frac{|\sph{d-1}|}{|\sph{d}|}\int_{0}^{\pi}v(\cos\theta)\:\Tsph{\theta}(f;\PT{x})\:(\sin\theta)^{d-1}\IntD{\theta}.
\end{equation}

\subsection{Preliminaries}\label{subsec:RFoS.pre-sph}
The restriction to $\sph{d}$ of a homogeneous and harmonic polynomial of degree $\ell$ on $\mathbb{R}^{d+1}$ is called a spherical harmonic of degree $\ell$ on $\sph{d}$. The collection of all spherical harmonics of degree $\ell$ on $\sph{d}$ is denoted by $\shSp{\ell}$.
Let $\Lp{p}{d}$, $1\le p<\infty$ be the $\mathbb{L}_{p}$-function space on $\sph{d}$ with respect to the normalised surface measure $\sigma_{d}$, and $\Lp{\infty}{d}$ be the collection of all continuous functions on $\sph{d}$.
The direct sum of all $\shSp{\ell}$, $\ell=0,1,2,\dots$, is dense in $\Lp{p}{d}$ for $1\le p\le\infty$, see e.g. \cite[Chapter~1]{WaLi2006}.
Given by \eqref{eq:RFoS.intro-preSph-1}, $\prj{\ell}$ denotes the projection operator on $\shSp{\ell}$.

Let $B$ be a Banach space embedded in $\Lp{1}{d}$. The modulus of continuity (or the second order modulus of smoothness) of $f\in B$ is defined by
\begin{equation}\label{eq:RFoS.modulus.continuity}
  \modu[B]{f;u}:=\sup_{0<\theta\le u}\normb{f-\Tsph{\theta}(f)}{B}, \quad 0<u\le\pi.
\end{equation}

Since $\normb{f-\Tsph{\theta}(f)}{\Lp{p}{d}}\to 0$ as $\theta\to 0^{+}$ for $1\le p\le\infty$, see e.g. \cite[p.~227, Lemma~4.2.2]{BeBuPa1968},
\begin{equation}\label{eq:RFoS.continuity.of.mod.conti}
  \modu{f;u}\to 0,\quad u\to 0^{+}.
\end{equation}

Let $\LBo$ denote the Laplace-Beltrami operator on $\sph{d}$. The $K$-functional of order $2$ on $\sph{d}$ is defined by
\begin{equation*}
 \Kf{f,t}:=\inf\left\{\norm{f-\varphi}{\Lp{p}{d}}+t\norm{\LBo \varphi}{\Lp{p}{d}}: \varphi\in \sob{p}{2}{d}\right\}.
\end{equation*}

The $K$-functional and the modulus of continuity for $\Lp{p}{d}$ are equivalent, see e.g. \cite[Theorem~5.1.2, p.~194]{WaLi2006}, \cite[Eq.~5.2, p.~95]{BeDaDi2003}:
\begin{equation}\label{eq:RFoS.pre-7}
    \Kf{f,\theta^{2}} \asymp \modu{f,\theta},\quad 0<\theta\le\pi,
\end{equation}
for $f\in \Lp{p}{d}$, $1\le p\le\infty$, where the constants in the inequalities depend only on $d$ and $p$.

Another key factor in the proof is an estimate for the translation operator. The translation $\Tsph[(d)]{\theta}$ is a strong $(p,p)$-type operator with operator norm $1$, see e.g. \cite[Theorem~2.4.1, p.~57]{WaLi2006}, \cite[Eq.~2.4.11, p.~237]{BeBuPa1968}, i.e. for $d\ge2$ and $1\le p\le \infty$,
\begin{equation}\label{eq:RFoS.T.norm}
  \normb{\Tsph[(d)]{\theta}}{L_{p}\rightarrow L_{p}}=1,\quad 0<\theta<\pi.
\end{equation}
We need the following upper bound for the difference between two translation operators.
\begin{lemma}\label{lm:RFoS.UB.T.sph}
Let $d\ge2$ and $1\le p\le\infty$. For any $f\in \Lp{p}{d}$, there exists a constant $c$ such that for $\psi,\theta>0$ and $0<\theta<\psi<\pi$,
\begin{equation*}
  \normb{\Tsph[(d)]{\psi}(f)-\Tsph[(d)]{\theta}(f)}{\Lp{p}{d}}\le c\;\modu{f,\sqrt{(\psi-\theta)(\psi+\theta)}},
\end{equation*}
where the constant $c$ depends only on $d$ and $p$.
\end{lemma}
\begin{remark}
    This upper bound is a generalisation of Theorem~5.1 of \cite{BeDaDi2003}, where the result is proved for the case when $\theta=0$.
\end{remark}
\begin{proof}[Proof of Lemma~\ref{lm:RFoS.UB.T.sph}.]
From \eqref{eq:RFoS.pre-7} and $\Tsph{\theta}(f;\PT{x}):=\Tsph{\pi-\theta}(f;-\PT{x})$, $\mathbf{x}\in\sph{d}$ (see e.g. \cite[Chapter~1]{WaLi2006}), we only need to prove for $0<\psi\le\pi/2$,
\begin{equation*}
  \norm{\Tsph{\psi}(f)-\Tsph{\theta}(f)}{\Lp{p}{d}}\le c_{d,p}\:\Kf{f,(\psi-\theta)(\psi+\theta)}.
\end{equation*}
For a spherical cap $\scap{\PT{x},u}\subset \sph{d}$, let $B_{u}$ be the spherical cap average
\begin{equation*}
    B_{u}\left(f; \PT{x}\right):=\frac{1}{|\scap{\PT{x},u}|}\int_{\scap{\PT{x},u}}f(\PT{y})\IntDiff{y},
\end{equation*}
where $|\scap{\PT{x},u}|$ is the measure of the cap $\scap{\PT{x},u}$. We shall also need the well known property
\begin{equation}\label{eq:RFoS.lm:UpperBDTransl-1}
  |\scap{\PT{x},u}|\asymp u^{d}.
\end{equation}
By the relation between the spherical cap average and the translation operator on the sphere, see \cite[Eq.~4.2.14, p.~238]{BeBuPa1968},
\begin{equation*}
  \Tsph{\theta}(\varphi;\PT{x}) - \varphi(\PT{x})=\frac{1}{|\sph{d-1}|}\int_{0}^{\theta}\frac{|\scap{\PT{x},u}|}{(\sin u)^{d-1}}B_{u}\left(\LBo \varphi; \PT{x}\right) \IntD{u}, \quad \varphi\in \sob{1}{2}{d},
\end{equation*}
we have for each $\PT{x}\in \sph{d}$ and $\varphi\in \sob{1}{2}{d}$,
\begin{equation*}
 \Tsph{\psi}(\varphi;\PT{x})-\Tsph{\theta}(\varphi;\PT{x}) = \frac{1}{|\sph{d-1}|}\int_{\theta}^{\psi}\frac{|\scap{\PT{x},u}|}{(\sin u)^{d-1}}B_{u}\left(\LBo \varphi; \PT{x}\right) \IntD{u}.
\end{equation*}
From \eqref{eq:RFoS.lm:UpperBDTransl-1} and $\|B_{u}\|_{L_{p}\rightarrow L_{p}}=1$, see e.g. \cite[Theorem~2.4.2, p.~59]{WaLi2006}, \cite[Eq.~4.2.4, p.~236]{BeBuPa1968}, for $1\le p \le \infty$,
\begin{align}\label{eq:RFoS.lm:UpperBDTransl-2}
 \norm{\Tsph{\psi}(\varphi)-\Tsph{\theta}(\varphi)}{\Lp{p}{d}}
 & \le
 \frac{1}{|\sph{d-1}|}\int_{\theta}^{\psi}\frac{|\scap{\PT{x},u}|}{(\sin u)^{d-1}}\normb{B_{u}\left(\LBo \varphi\right)}{\Lp{p}{d}} \IntD{u}\nonumber\\[2mm]
 & \le c_{d}\:\normb{\LBo \varphi}{\Lp{p}{d}} \int_{\theta}^{\psi}u\: \IntD{u}\notag\\[2mm]
 & \le c_{d}^{\prime}\:(\psi-\theta)(\psi+\theta)\:\normb{\LBo \varphi}{\Lp{p}{d}}.
\end{align}
By \eqref{eq:RFoS.T.norm} we obtain for $f\in \Lp{p}{d}$ and any $\varphi\in \sob{p}{2}{d}$,
\begin{align*}
 \|\Tsph{\psi}(f)-\Tsph{\theta}(f)\|_{\Lp{p}{d}}
 & =\norm{\Tsph{\psi}(f-\varphi)-\Tsph{\theta}(f-\varphi)+\Tsph{\psi}(\varphi)-\Tsph{\theta}(\varphi)}{\Lp{p}{d}}\\[2mm]
 &\le 2\norm{f-\varphi}{\Lp{p}{d}}+c_{d}^{\prime}\:(\psi-\theta)(\psi+\theta)\:\normb{\LBo \varphi}{\Lp{p}{d}}
 \end{align*}
 which with an optimal choice of $\varphi$ gives, with new constants $c_{d}$ and $c_{d,p}$,
 \begin{align*}
 \norm{\Tsph{\psi}(f) - \Tsph{\theta}(f)}{\Lp{p}{d}}
 &\le c_{d}\:\Kf{f,(\psi-\theta)(\psi+\theta)}\\
 &\le c_{d,p}\:\modu{f,\sqrt{(\psi-\theta)(\psi+\theta)}}.
\end{align*}
This completes the proof.
\end{proof}

\subsection{Upper bounds}
\begin{theorem}\label{thm:RFoS.upper.bound.local.conv.sphere} Let $d$ be an integer and $p,\delta$ be real numbers satisfying $d\ge 2$, $1\le p\le\infty$ and $0<\delta<\pi$. For $f\in \Lp{p}{d}$,
\begin{equation}\label{eq:RFoS.thm:UpperBdSph-15}
  \normb{\Vdl{L}(f)}{\Lp{p}{d}}\le
  c\:L^{\frac{d-1}{2}}\left(
  L^{-1}\norm{f}{\Lp{p}{d}}+\:\modub{f,L^{-\frac{1}{2}}}\right),
\end{equation}
where the constant $c$ depends only on $d$, $p$ and $\delta$.
\end{theorem}
The proof of Theorem~\ref{thm:RFoS.upper.bound.local.conv.sphere} is given later in this section.
\begin{remark} From Theorem~\ref{thm:RFoS.upper.bound.local.conv.sphere}, if $f$ is a Lipschitz function, then
\begin{equation*}
  \|\Vdl{L}(f)\|_{\Lp{p}{d}}\le c_{d,p,\delta}\:
  L^{\frac{d-2}{2}}\left(L^{-\frac{1}{2}}\|f\|_{\Lp{p}{d}}+c_{f} \right),\quad d\ge2.
\end{equation*}
If $f\in \sob{p}{2}{d}$, then
\begin{equation*}
\modub{f,L^{-\frac{1}{2}}}
\asymp \Kf{f,L^{-1}}
\asymp \norm{\Tsph{1/\sqrt{L}}(f)-f}{\Lp{p}{d}}\le c_{d,p} \:L^{-1} \norm{\LBo f}{\Lp{p}{d}},
\end{equation*}
where the first equivalence is from \eqref{eq:RFoS.pre-7}, the second is by \cite[Theorem~5.1, p.~94]{BeDaDi2003} and the last inequality is by \eqref{eq:RFoS.lm:UpperBDTransl-2} with $\theta=0$ and $\phi=L^{-\frac{1}{2}}$. Hence,
\begin{equation}\label{eq:RFoS.rmk:UpperBdSph-3}
  \normb{\Vdl{L}(f)}{\Lp{p}{d}}\le
  c_{d,p,\delta}\:L^{\frac{d-3}{2}}\left(\norm{f}{\Lp{p}{d}}+\norm{\LBo f}{\Lp{p}{d}}\right),\quad d\ge2.
\end{equation}
\end{remark}
Since $\sob{p}{r}{d}\subset \sob{p}{s}{d}$ for $0\le s\le r<\infty$ and by \eqref{eq:RFoS.rmk:UpperBdSph-3}, we have the following upper bound for the Fourier local convolutions with sufficiently smooth functions.
\begin{corollary}\label{cor:RFoS.UpperBDSph-2} Let $d\ge 2$, $s\ge2$, $1\le p\le\infty$ and $0<\delta<\pi$. Then, for $f\in \sob{p}{s}{d}$,
\begin{equation*}
  \normb{\Vdl{L}(f)}{\Lp{p}{d}}\le
  c\:L^{\frac{d-3}{2}}\norm{f}{\sob{p}{s}{d}},
\end{equation*}
where the constant $c$ depends only on $d$, $p$, $s$ and $\delta$.
\end{corollary}
\begin{remark}
The corollary implies that the Fourier convolution has the Riemann localisation property for $\sob{p}{s}{2}$ and $s\ge 2$. For higher dimensional spheres $\sph{d}$ with $d\ge3$, however, the Fourier convolution does not have the Riemann localisation property in general, as will be shown in Theorem~\ref{thm:RFoS.lower.bound.loc.Fourier.sph.const}.
\end{remark}
That the translation operator commutes with the Laplace-Beltrami operator enables us to replace the $\mathbb{L}_{p}$-norms in inequalities \eqref{eq:RFoS.thm:UpperBdSph-15} and \eqref{eq:RFoS.rmk:UpperBdSph-3} by Sobolev norms.
\begin{theorem}\label{thm:RFoS.Local.Fourier.conv.sphere.UB.Sobolev} Let $d\ge 2$, $s\ge0$, $1\le p\le\infty$ and $0<\delta<\pi$. Then, for $f\in \sob{p}{s}{d}$,
\begin{equation*}
  \normb{\Vdl{L}(f)}{\sob{p}{s}{d}}\le
  c\:L^{\frac{d-1}{2}}\left(
  L^{-1}\norm{f}{\sob{p}{s}{d}}+\:\modub[{\sob{p}{s}{d}}]{f,L^{-\frac{1}{2}}}\right).
\end{equation*}
For $f\in \sob{p}{s+2}{d}$,
\begin{equation*}
  \normb{\Vdl{L}(f)}{\sob{p}{s}{d}}\le
  c\:L^{\frac{d-3}{2}}\left(
  \norm{f}{\sob{p}{s}{d}}+\:\norm{\LBo f}{\sob{p}{s}{d}}\right).
\end{equation*}
Here, the constants $c$ depend only on $d$, $p$, $s$ and $\delta$.
\end{theorem}

We only give the proof of Theorem~\ref{thm:RFoS.upper.bound.local.conv.sphere}. The proof of the first part of Theorem~\ref{thm:RFoS.Local.Fourier.conv.sphere.UB.Sobolev} is similar.
\begin{proof}[Proof of Theorem~\ref{thm:RFoS.upper.bound.local.conv.sphere}.] Since the proof for $\frac{\pi}{2}\le \delta<\pi$ can be deduced from that for $0<\delta<\frac{\pi}{2}$, we only consider the latter case. Let $\PT{x}\in \sph{d}$. Then by \eqref{eq:RFoS.integral.zonal.via.translation} we have
\begin{align*}
    \Vdl{L}(f;\PT{x})= \int_{\sph{d}\backslash \scap{\PT{x},\delta}}\vd{L}(\PT{x}\cdot \PT{y})f(\PT{y})\IntDiff{y}
                   = \frac{|\sph{d-1}|}{|\sph{d}|}\int_{\delta}^{\pi}\vd{L}(\cos\theta)\:\Tsph[(d)]{\theta}(f;\PT{x})(\sin\theta)^{d-1} \IntD{\theta}.
\end{align*}
Splitting the integral, we have
\begin{equation}\label{eq:RFoS.thm:upper.bound.local.conv.sphere-9}
   \frac{|\sph{d}|}{|\sph{d-1}|} \Vdl{L}(f;\PT{x})
                   = \left(\int_{\delta}^{\frac{\pi}{2}}+\int_{\frac{\pi}{2}}^{\pi}\right)\Tsph{\theta}(f;\PT{x})\:\vd{L}(\cos\theta)\:(\sin\theta)^{d-1}\IntD{\theta}
                   =: I_{1}(f;\PT{x})+I_{2}(f;\PT{x}).
\end{equation}
For $I_{1}(f;\PT{x})$, applying \eqref{eq:RFoS.Dirichlet.Jacobi.one.term.asymp-a} of Lemma~\ref{lm:RFoS.Dirichlet.Jacobi.one.term.asymp} with $\alpha=\beta=\frac{d-2}{2}$ and hence $\Lshift=L+\frac{d}{2}$, and by Lemma~\ref{lm:prefisph.vd.vab}, we have
\begin{align}\label{eq:RFoS.thm:upper.bound.local.conv.sphere-5}
  I_{1}(f;\PT{x})
       &= \int_{\delta}^{\frac{\pi}{2}}\Tsph{\theta}(f;\PT{x})\frac{2^{-(d-1)}}{\Gamma(\frac{d+1}{2})}\Lshift^{\frac{d-1}{2}}\Bigl[\bigl(\sin\tfrac{\theta}{2}\bigr)^{-\frac{d+1}{2}}\bigl(\cos\tfrac{\theta}{2}\bigr)^{-\frac{d-1}{2}}\sin(\thetan)+\bigo{d,\delta}{L^{-1}}\Bigr]\notag\\
       &\qquad\times(\sin\theta)^{d-1}\IntD{\theta}\notag\\
       &\hspace{-1cm}= \frac{\Lshift^{\frac{d-1}{2}}}{\Gamma(\frac{d+1}{2})}\left[\int_{\delta}^{\frac{\pi}{2}}\Tsph{\theta}(f;\PT{x})\bigl(\sin\tfrac{\theta}{2}\bigr)^{\frac{d-3}{2}}\bigl(\cos\tfrac{\theta}{2}\bigr)^{\frac{d-1}{2}}\sin(\thetan)\IntD{\theta}
       +\norm{f}{\Lp{1}{d}}\:\bigo{d,\delta}{L^{-1}}\right]\notag\\
       &\hspace{-1cm}=: \frac{\Lshift^{\frac{d-1}{2}}}{\Gamma(\frac{d+1}{2})}\Bigl[I_{1,1}(f;\PT{x})+ \norm{f}{\Lp{1}{d}}\:\bigo{d,\delta}{L^{-1}}\Bigr],
\end{align}
where \begin{equation}\label{eq:RFoS.tilde.u.theta.n}
    \thetan:=\thetan[\theta,L;d]:=\left(L+\tfrac{d}{2}\right)\theta-\tfrac{d-1}{4}\pi
\end{equation}
and we used
\begin{equation}\label{eq:RFoS.Tf.UB.f.L1}
    \biggl|\int_{\delta}^{\frac{\pi}{2}} \Tsph{\theta}(f;\PT{x})\: (\sin\theta)^{d-1} \IntD{\theta}\biggr|
      \le \int_{0}^{\pi} \Tsph{\theta}(|f|;\PT{x})\: (\sin\theta)^{d-1}\IntD{\theta}
      = \norm{f}{\Lp{1}{d}}.
\end{equation}

Next, we estimate the $\mathbb{L}_{p}$-norm of $I_{1,1}(f)$ in \eqref{eq:RFoS.thm:upper.bound.local.conv.sphere-5}.
Let $m_{1}(\theta):=\bigl(\sin\frac{\theta}{2}\bigr)^{\frac{d-3}{2}}\bigl(\cos\frac{\theta}{2}\bigr)^{\frac{d-1}{2}}$ and $\theta_{L}:=\pi/\Lshift$, $\gamma:=\frac{d-1}{4}\pi$, where $\Lshift=L+\frac{d}{2}$. By \eqref{eq:RFoS.tilde.u.theta.n}, $\thetan=\Lshift\theta-\gamma$, then,
\begin{align}\label{eq:RFoS.UB.I11.tf}
  I_{1,1}(f;\PT{x})
  &:= \int_{\delta}^{\frac{\pi}{2}}\Tsph{\theta}(f;\PT{x})\:m_{1}(\theta)\sin(\Lshift\theta-\gamma) \IntD{\theta} \notag\\
  &= -\int_{\delta-\theta_{L}}^{\frac{\pi}{2}-\theta_{L}}\Tsph{t+\theta_{L}}(f;\PT{x})\:m_{1}(t+\theta_{L})\sin(\Lshift t-\gamma) \IntD{t} \notag\\
  &= \frac{1}{2}\int_{\delta}^{\frac{\pi}{2}-\theta_{L}}\left[\Tsph{t}(f;\PT{x})-\Tsph{t+\theta_{L}}(f;\PT{x})\right]m_{1}(t)\sin(\Lshift t-\gamma) \IntD{t} + I_{1,1,1}(f;\PT{x}),
\end{align}
where
\begin{align*}
  I_{1,1,1}(f;\PT{x}) &:=  \frac{1}{2}\biggl[\int_{\frac{\pi}{2}-\theta_{L}}^{\frac{\pi}{2}}\Tsph{\theta}(f;\PT{x})m_{1}(\theta)\sin(\Lshift\theta-\gamma)\IntD{\theta}\\
               &\qquad -\int_{\delta-\theta_{L}}^{\delta}\Tsph{t+\theta_{L}}(f;\PT{x})\:m_{1}(t+\theta_{L})\sin(\Lshift t-\gamma)\IntD{t}\\
                &\qquad+\int_{\delta}^{\frac{\pi}{2}-\theta_{L}}\Tsph{t+\theta_{L}}(f;\PT{x})\:\bigl(m_{1}(t)-m_{1}(t+\theta_{L})\bigr)\sin(\Lshift t-\gamma)\IntD{t}\biggr].
\end{align*}

Since $\Tsph{\theta}$ in $\Lp{p}{d}$ is bounded with norm $1$, see \eqref{eq:RFoS.T.norm}, and the derivative of $m_{1}(\theta)$ is bounded over $[\delta,\pi/2]$, it follows that
\begin{equation*}
  \normb{I_{1,1,1}(f)}{\Lp{p}{d}} \le c_{d,\delta}\: L^{-1} \norm{f}{\Lp{p}{d}}.
\end{equation*}
Applying Lemma~\ref{lm:RFoS.UB.T.sph} to the integral of the last equality in \eqref{eq:RFoS.UB.I11.tf} then gives
\begin{equation}\label{eq:I11.Lp.nrm}
  \norm{I_{1,1}(f)}{\Lp{p}{d}}
  \le c_{d,p,\delta}\: \left(L^{-1} \norm{f}{\Lp{p}{d}}+\modub{f,L^{-\frac{1}{2}}}\right).
\end{equation}
This with \eqref{eq:RFoS.thm:upper.bound.local.conv.sphere-5} gives
\begin{equation}\label{eq:RFoS.thm:UpperBdSph-13}
\norm{I_{1}(f)}{\Lp{p}{d}}\le c_{d,p,\delta}\:L^{\frac{d-1}{2}}\left(L^{-1}\norm{f}{\Lp{p}{d}}+\modub{f,L^{-\frac{1}{2}}}\right).
\end{equation}
This finishes the estimate of $I_{1}$.

We have an analogous proof for $I_{2}$. Let $k_{0}$ be a positive integer (independent of $L$) such that $\xi_{0}:=\xi_{0}(L):=(k_{0}\pi+\frac{d-1}{4}\pi)/(L+\frac{d}{2})>c^{(1)}L^{-1}$ for all positive integers $L$, where $c^{(1)}$ is the constant in Lemmas~\ref{lm:RFoS.Dirichlet.Jacobi.one.term.asymp} and \ref{lm:RFoS.Dirichlet.Jacobi.kernel.upper.bound} with $\alpha=\beta=\frac{d-2}{2}$. Then,
\begin{eqnarray}
  I_{2}(f;\PT{x})&=& \int_{\frac{\pi}{2}}^{\pi}\Tsph{\theta}(f;\PT{x})\:\vd{L}(\cos\theta)\:(\sin\theta)^{d-1}\IntD{\theta}\nonumber\\
       &=&\left(\int_{\frac{\pi}{2}}^{\pi-\xi_{0}}+\int_{\pi-\xi_{0}}^{\pi}\right)\Tsph{\theta}(f;\PT{x})\:\vd{L}(\cos\theta)\:(\sin\theta)^{d-1}\IntD{\theta}\nonumber\\
       &=:& I_{2,1}(f;\PT{x})+ I_{2,2}(f;\PT{x}).\label{eq:RFoS.thm:upper.bound.local.conv.sphere-8}
\end{eqnarray}
For $I_{2,1}(f;\PT{x})$, applying \eqref{eq:RFoS.Dirichlet.Jacobi.one.term.asymp-b} of Lemma~\ref{lm:RFoS.Dirichlet.Jacobi.one.term.asymp} with the substitution $\theta^{\prime}=\pi-\theta$ and by Lemma~\ref{lm:prefisph.vd.vab}, cf. \eqref{eq:RFoS.thm:upper.bound.local.conv.sphere-5},
\begin{align}
  &\hspace{-0.5cm}I_{2,1}(f;\PT{x})
= \int_{\frac{\pi}{2}}^{\pi-\xi_{0}}\Tsph{\theta}(f;\PT{x})\vd{L}(\cos\theta)(\sin\theta)^{d-1}\IntD{\theta}\nonumber\\
       &\hspace{1cm}= \frac{(-1)^{L}\:\Lshift^{\frac{d-1}{2}}}{\Gamma(\frac{d+1}{2})}\biggl[\int_{\xi_{0}}^{\frac{\pi}{2}}\Tsph{\theta}(f;-\PT{x})\bigl(\sin\tfrac{\theta}{2}\bigr)^{\frac{d-1}{2}}\bigl(\cos\tfrac{\theta}{2}\bigr)^{\frac{d-3}{2}}\sin(\thetan+\tfrac{\pi}{2})\IntD{\theta}\nonumber\\
       &\hspace{5.5cm}+\mathcal{O}_{d}\left(L^{-1}\right)\int_{\xi_{0}}^{\frac{\pi}{2}}|\Tsph{\theta}(f;-\PT{x})|\bigl(\sin\tfrac{\theta}{2}\bigr)^{\frac{d-3}{2}}\IntD{\theta}\biggr],\notag\\
&\hspace{1cm}=\frac{(-1)^{L}\:\Lshift^{\frac{d-1}{2}}}{\Gamma(\frac{d+1}{2})}
\biggl[I_{2,1,1}(f;\PT{x}) + \bigo{d}{L^{-1}}\norm{f}{\Lp{p}{d}}\biggr],\label{eq:RFoS.thm:upper.bound.local.conv.sphere-10}
\end{align}
where $\thetan$ is given by \eqref{eq:RFoS.tilde.u.theta.n} and the second term of the last equality can be proved in a similar way to \eqref{eq:RFoS.Tf.UB.f.L1}.

The term $I_{2,1,1}(f;\PT{x})$ in \eqref{eq:RFoS.thm:upper.bound.local.conv.sphere-10} can be estimated similarly to $I_{1,1}(f;\PT{x})$. Let $m_{2}(\theta):=\bigl(\sin\frac{\theta}{2}\bigr)^{\frac{d-1}{2}}\bigl(\cos\frac{\theta}{2}\bigr)^{\frac{d-3}{2}}$ and $\theta_{L}:=\pi/\Lshift$, $\gamma':=\frac{d-1}{4}\pi$, where $\Lshift=L+\frac{d}{2}$. By \eqref{eq:RFoS.tilde.u.theta.n}, $\thetan+\frac{\pi}{2}=\Lshift\theta-\gamma'$, then,
\begin{align}\label{eq:RFoS.UB.I211.tf}
  I_{2,1,1}(f;\PT{x})
  &:= \int_{\xi_{0}}^{\frac{\pi}{2}}\Tsph{\theta}(f;\PT{x})\:m_{2}(\theta)\sin(\Lshift\theta-\gamma') \IntD{\theta} \notag\\
  &= -\int_{\xi_{0}-\theta_{L}}^{\frac{\pi}{2}-\theta_{L}}\Tsph{t+\theta_{L}}(f;\PT{x})\:m_{2}(t+\theta_{L})\sin(\Lshift t-\gamma') \IntD{t} \notag\\
  &= \frac{1}{2}\int_{\xi_{0}}^{\frac{\pi}{2}-\theta_{L}}\left[\Tsph{t}(f;\PT{x})-\Tsph{t+\theta_{L}}(f;\PT{x})\right]m_{2}(t)\sin(\Lshift t-\gamma') \IntD{t} + I_{2,1,2}(f;\PT{x}),
\end{align}
where
\begin{align*}
  I_{2,1,2}(f;\PT{x}) &:=  \frac{1}{2}\biggl[\int_{\frac{\pi}{2}-\theta_{L}}^{\frac{\pi}{2}}\Tsph{\theta}(f;\PT{x})m_{2}(\theta)\sin(\Lshift\theta-\gamma')\IntD{\theta}\\
               &\qquad -\int_{\xi_{0}-\theta_{L}}^{\xi_{0}}\Tsph{t+\theta_{L}}(f;\PT{x})\:m_{2}(t+\theta_{L})\sin(\Lshift t-\gamma')\IntD{t}\\
                &\qquad+\int_{\xi_{0}}^{\frac{\pi}{2}-\theta_{L}}\Tsph{t+\theta_{L}}(f;\PT{x})\:\bigl(m_{2}(t)-m_{2}(t+\theta_{L})\bigr)\sin(\Lshift t-\gamma')\IntD{t}\biggr].
\end{align*}
By \eqref{eq:RFoS.T.norm}, since $\Bigl|\Diff[\theta]{1} m_{2}(\theta)\Bigr|\le c \max\left\{\theta^{\frac{d-3}{2}},1\right\}$, $0<\theta\le \pi/2$,
\begin{equation*}
  \normb{I_{2,1,2}(f)}{\Lp{p}{d}} \le c_{d}\: L^{-1} \norm{f}{\Lp{p}{d}}.
\end{equation*}
Applying Lemma~\ref{lm:RFoS.UB.T.sph} to the integral of the last equality in \eqref{eq:RFoS.UB.I211.tf} then gives
\begin{equation}\label{eq:I211.Lp.nrm}
    \normb{I_{2,1,1}(f)}{\Lp{p}{d}}
    \le c_{d,p}\: \left(L^{-1} \norm{f}{\Lp{p}{d}}+\modub{f,L^{-\frac{1}{2}}}\right).
\end{equation}
This with \eqref{eq:RFoS.thm:upper.bound.local.conv.sphere-10} gives
\begin{equation}\label{eq:RFoS.thm:UpperBdSph-14}
  \norm{I_{2,1}(f)}{\Lp{p}{d}}\le
  c_{d,p}\:L^{\frac{d-1}{2}}\left(
  L^{-1}\norm{f}{\Lp{p}{d}}+\:\modub{f,L^{-\frac{1}{2}}}\right).
\end{equation}
For $I_{2,2}(f)$, using \eqref{eq:RFoS.Dirichlet.Jacobi.upper.bound-b} of Lemma~\ref{lm:RFoS.Dirichlet.Jacobi.kernel.upper.bound} with $\alpha=\beta=\frac{d-2}{2}$, we have
\begin{equation}
  \norm{I_{2,2}(f)}{\Lp{p}{d}}
       \le c_{d,p}\int_{\pi-c^{(1)}L^{-1}}^{\pi}\hspace{-1mm}\norm{\Tsph{\theta}(f;\cdot)}{\Lp{p}{d}}L^{d-1}(\sin\theta)^{d-1}\IntD{\theta}
       \le c_{d,p}\:L^{-1}\norm{f}{\Lp{p}{d}}.\label{eq:RFoS.thm:UpperBdSph-1}
\end{equation}
The combination of \eqref{eq:RFoS.thm:UpperBdSph-1}, \eqref{eq:RFoS.thm:UpperBdSph-14}, \eqref{eq:RFoS.thm:upper.bound.local.conv.sphere-8}, \eqref{eq:RFoS.thm:UpperBdSph-13} and \eqref{eq:RFoS.thm:upper.bound.local.conv.sphere-9} gives \eqref{eq:RFoS.thm:UpperBdSph-15}.
\end{proof}

\subsection{Lower bounds}
In this section, we show a lower bound of the $\mathbb{L}_{p}$-norm of the local convolution for a constant function on the sphere $\sph{d}$, $d\ge2$. This lower bound matches the upper bound of the local convolution for Sobolev space $\sob{p}{s}{d}$ with $s\ge2$, see Corollary~\ref{cor:RFoS.UpperBDSph-2}. It thus establishes that the upper bound for the local convolution for these Sobolev spaces is optimal.
\begin{theorem}\label{thm:RFoS.lower.bound.loc.Fourier.sph.const} Let $d\ge2$, $1\le p\le \infty$ and $0<\delta<\pi/2$. Then there exists a subsequence $\{L_{\ell}\}_{\ell\ge1}\subset\Zp$ such that for $\ell\ge1$,
\begin{equation*}
  \normau{\Vdl{L_{\ell}}(\mathbf{1})}{\Lp{p}{d}} \ge c\: L_{\ell}^{\frac{d-3}{2}},
\end{equation*}
where the constant $c$ depends only on $d$ and $\delta$.
\end{theorem}

\begin{proof}[Proof]
Let $\PT{x}\in\sph{d}$. Then
\begin{align*}
    \Vdl{L}(\mathbf{1};\PT{x})
    & = \int_{\sph{d}\backslash \scap{\PT{x},\delta}} \mathbf{1}(\PT{y}) \vd{L}(\PT{x}\cdot\PT{y}) \IntDiff{y}\\[2mm]
    & = \frac{|\sph{d-1}|}{|\sph{d}|}\int_{\delta}^{\pi} \vd{L}(\cos\theta) (\sin\theta)^{d-1} \IntD{\theta}\\[1mm]
    & = \frac{|\sph{d-1}|}{|\sph{d}|} \left(\int_{\delta}^{\pi-c^{(1)}L^{-1}}+\int_{\pi-c^{(1)}L^{-1}}^{\pi}\right) \vd{L}(\cos\theta) (\sin\theta)^{d-1} \IntD{\theta}\\
    & = \frac{|\sph{d-1}|}{|\sph{d}|} \int_{\delta}^{\pi-c^{(1)}L^{-1}}\hspace{-5mm} \vd{L}(\cos\theta) (\sin\theta)^{d-1} \IntD{\theta} + \bigo{d}{L^{-1}},
\end{align*}
where $c^{(1)}$ is the constant from Lemmas~\ref{lm:RFoS.Dirichlet.Jacobi.two.term.asymp} and \ref{lm:RFoS.Dirichlet.Jacobi.kernel.upper.bound}, and the last line uses Lemma~\ref{lm:prefisph.vd.vab} and \eqref{eq:RFoS.Dirichlet.Jacobi.upper.bound-b} of Lemma~\ref{lm:RFoS.Dirichlet.Jacobi.kernel.upper.bound}.
Using Lemma~\ref{lm:prefisph.vd.vab} again gives
\begin{equation*}
  \Vdl{L}(\mathbf{1};\PT{x})= \int_{\delta}^{\pi-c^{(1)}L^{-1}}\hspace{-3mm}\vab[\frac{d-2}{2},\frac{d-2}{2}]{L}(1,\cos\theta)\:(\sin\theta)^{d-1}\IntD{\theta} + \bigo{d}{L^{-1}}.
\end{equation*}
We now apply \eqref{eq:RFoS.Dirichlet.Jacobi.two.term.asymp-b} of Lemma~\ref{lm:RFoS.Dirichlet.Jacobi.two.term.asymp} ii) with $\alpha=\beta=\frac{d-2}{2}$ and hence $\Lshift=L+\frac{d}{2}$ and then take the substitution $\theta^{\prime}=\pi-\theta$. Then
\begin{align}\label{eq:RFoS.lower.bound.loc.Fourier.sph-Vdl-d2}
  &\Vdl{L}(\mathbf{1};\PT{x}) =   \frac{(-1)^{L}}{2^{d-1}\Gamma(\frac{d}{2})}\:\Lshift^{\frac{d-1}{2}}\int_{c^{(1)}L^{-1}}^{\pi-\delta} \hspace{-2mm}m_{\frac{d-2}{2},\frac{d}{2}}(\theta)\Bigl[\cos\omega_{\frac{d-2}{2}}(\Lshift\theta)\notag\\[1mm]
       &\hspace{1mm} + \Lshift^{-1} F^{(4)}_{\frac{d-2}{2},\frac{d-2}{2}}(\Lshift,\theta)+\bigo{d,\delta}{L^{\widehat{u}(\frac{d-2}{2})}\theta^{\widehat{\nu}(\frac{d-2}{2})}}+\bigo{d}{L^{-2}\theta^{-2}}\Bigr](\sin\theta)^{d-1}\IntD{\theta}+\bigo{d}{L^{-1}}\nonumber\\[2mm]
       &\hspace{-0.1cm} = \frac{(-1)^{L}}{\sqrt{\pi}\:\Gamma(\frac{d}{2})}\:\Lshift^{\frac{d-1}{2}}
       \biggl[\int_{c^{(1)}L^{-1}}^{\pi-\delta}\hspace{-2mm}(\sin\tfrac{\theta}{2})^{\frac{d-1}{2}}(\cos\tfrac{\theta}{2})^{\frac{d-3}{2}}\cos\untheta{\theta}\IntD{\theta}\notag\\[1mm]
       &\hspace{5mm} +\Lshift^{-1} \int_{c^{(1)}L^{-1}}^{\pi-\delta}\hspace{-2mm}(\sin\tfrac{\theta}{2})^{\frac{d-1}{2}}(\cos\tfrac{\theta}{2})^{\frac{d-3}{2}}F_{\frac{d-2}{2},\frac{d-2}{2}}^{(4)}(\Lshift,\theta)\IntD{\theta}
       +\bigo{d,\delta}{L^{\widehat{u}(\frac{d-2}{2})}}\notag\\[1mm]
       &\hspace{5mm}  + \bigo{d}{L^{-2}} \int_{c^{(1)}L^{-1}}^{\pi-\delta}\hspace{-2mm}\theta^{-2}(\sin\tfrac{\theta}{2})^{\frac{d-1}{2}}(\cos\tfrac{\theta}{2})^{\frac{d-3}{2}}\IntD{\theta}\biggr] + \bigo{d}{L^{-1}},
\end{align}
where $\widehat{u}(\frac{d-2}{2})<-1$.

Since $\int_{c^{(1)}L^{-1}}^{\pi-\delta}\theta^{-2}(\sin\tfrac{\theta}{2})^{\frac{d-1}{2}}(\cos\tfrac{\theta}{2})^{\frac{d-3}{2}}\IntD{\theta}=\bigo{d}{\sqrt{L}}$ for $d\ge2$, \eqref{eq:RFoS.lower.bound.loc.Fourier.sph-Vdl-d2} becomes
\begin{align}\label{eq:RFoS.lower.bound.loc.Fourier.sph-Vdl}
  \Vdl{L}(\mathbf{1};\PT{x})
  &= \frac{(-1)^{L}}{\sqrt{\pi}\:\Gamma(\frac{d}{2})}\:\Lshift^{\frac{d-1}{2}}
  \hspace{-1mm}\left[I_{1}+\Lshift^{-1}I_{2}+ \bigo{d,\delta}{L^{\widehat{u}(\frac{d-2}{2})}} + \bigo{d}{L^{-\frac{3}{2}}}\right]+\bigo{d}{L^{-1}}\notag\\[1mm]
  &= \frac{(-1)^{L}}{\sqrt{\pi}\:\Gamma(\frac{d}{2})}\:\Lshift^{\frac{d-1}{2}}
  \hspace{-1mm}\left[I_{1}+\Lshift^{-1}I_{2}+ \bigo{d,\delta}{L^{\widehat{u}(\frac{d-2}{2})}} + \bigo{d}{L^{-\frac{3}{2}}}\right],
\end{align}
where
\begin{align*}
    I_{1} &:= \int_{c^{(1)}L^{-1}}^{\pi-\delta}\hspace{-2mm}(\sin\tfrac{\theta}{2})^{\frac{d-1}{2}}(\cos\tfrac{\theta}{2})^{\frac{d-3}{2}}\cos\untheta{\theta}\IntD{\theta},\\[1mm]
    I_{2} &:= \int_{c^{(1)}L^{-1}}^{\pi-\delta}\hspace{-2mm}(\sin\tfrac{\theta}{2})^{\frac{d-1}{2}}(\cos\tfrac{\theta}{2})^{\frac{d-3}{2}}F_{\frac{d-2}{2},\frac{d-2}{2}}^{(4)}(\Lshift,\theta)\IntD{\theta}.
\end{align*}

We will prove in the remaining part that $|I_{1}|$ is lower bounded by $c_{d,\delta}\:L_{\ell}^{-1}$ for a subsequence $L_{\ell}$ of $L$ and that $I_{2}=o(1)$ (so $\Lshift^{-1}I_{2}$ is a higher order term than $I_{1}$), while the two big $\mathcal{O}$ terms have smaller asymptotic orders. Thus, $I_{1}$ is the dominant term.
By \eqref{eq:RFoS.FA4} of Lemma~\ref{lm:RFoS.Dirichlet.Jacobi.two.term.asymp},
\begin{equation*}\label{eq:RFoS.lower.bound.loc.Fourier.sph-I2-1}
  \hspace{-1mm} I_{2} = \int_{c^{(1)}L^{-1}}^{\pi-\delta}\hspace{-2mm}(\sin\tfrac{\theta}{2})^{\frac{d-1}{2}}(\cos\tfrac{\theta}{2})^{\frac{d-3}{2}}\FB[\frac{d-2}{2},\frac{d}{2}]{\theta}\:\sin\untheta{\theta}\IntD{\theta}.
\end{equation*}
Since the function $(\sin\frac{\theta}{2})^{\frac{d-1}{2}}(\cos\frac{\theta}{2})^{\frac{d-3}{2}}\FB[\frac{d-2}{2},\frac{d}{2}]{\theta}$ is in $\mathbb{L}_{1}(0,\pi-\delta)$ for $d\ge2$, we may apply the Riemann-Lebesgue lemma to $I_{2}$. Thus
\begin{equation}\label{eq:RFoS.lower.bound.loc.Fourier.sph-I2-2}
  I_{2}\to 0\; \mbox{~as~}\; L\to \infty.
\end{equation}
For $I_{1}$ of \eqref{eq:RFoS.lower.bound.loc.Fourier.sph-Vdl}, let $B_{1}(\theta):=(\sin\frac{\theta}{2})^{\frac{d-1}{2}}(\cos\frac{\theta}{2})^{\frac{d-3}{2}}$. Using integration by parts,
\begin{align}\label{eq:RFoS.lower.bound.loc.Fourier.sph-I1}
  I_{1} &= \int_{c^{(1)}L^{-1}}^{\pi-\delta}\hspace{-2mm}B_{1}(\theta)\cos\untheta{\theta}\IntD{\theta}\nonumber\\[2mm]
  &= \frac{1}{L+\frac{d}{2}}\biggl[B_{1}(\pi-\delta)\sin\untheta{(\pi-\delta)}\notag\\[1mm]
  &\hspace{2.3cm}-B_{1}(c^{(1)}L^{-1})\sin\untheta{c^{(1)}L^{-1}} \notag\\[1mm]
  &\hspace{4cm}+\int_{c^{(1)}L^{-1}}^{\pi-\delta}\hspace{-2mm}B'_{1}(\theta)\sin\untheta{\theta}\IntD{\theta}\biggr]\nonumber\\[2mm]
  &=: \frac{1}{L+\frac{d}{2}} \left[I_{1,1} - \bigo{d}{L^{-\frac{1}{2}}} - I_{1,2}\right].
\end{align}
Since $B'_{1}(\theta)$ is in $\mathbb{L}_{1}(0,\pi-\delta)$, the Riemann-Lebesgue lemma gives
\begin{equation}\label{eq:RFoS.lower.bound.loc.Fourier.sph-I12}
  I_{1,2}\to 0\;\; \hbox{as} \;\; L\to \infty.
\end{equation}
For $I_{1,1}$ of \eqref{eq:RFoS.lower.bound.loc.Fourier.sph-I1},
\begin{align*}
  I_{1,1} &= B_{1}(\pi-\delta)\sin\untheta{(\pi-\delta)}\nonumber\\[2mm]
  &= (-1)^{L+1} (\sin\tfrac{\delta}{2})^{\frac{d-3}{2}}(\cos\tfrac{\delta}{2})^{\frac{d-1}{2}} \sin\left((L+\tfrac{d}{2})\delta-\tfrac{d+1}{4}\pi\right).
\end{align*}
Hence,
\begin{equation}\label{eq:RFoS.lower.bound.loc.Fourier.sph-I11}
  |I_{1,1}| = (\sin\tfrac{\delta}{2})^{\frac{d-3}{2}}(\cos\tfrac{\delta}{2})^{\frac{d-1}{2}} \Bigl|\sin\left((L+\tfrac{d}{2})\delta-\tfrac{d+1}{4}\pi\right)\Bigr|.
\end{equation}
Let $\xi$ be a positive real number in $(0,\pi/4)$ and let $c_{\xi}:=\sin\xi>0$. We want
\begin{equation*}
  \Bigl|\sin\bigl((L+\tfrac{d}{2})\delta-\tfrac{d+1}{4}\pi\bigr)\Bigr| > c_{\xi}.
\end{equation*}
This is equivalent to the assertion that $(L+\frac{d}{2})\delta-\frac{d+1}{4}\pi$ is in the interval $(k\pi+\xi, k\pi+\pi-\xi)$ for some integer $k$. That is, $L$ must fall into the interval $\mathcal{I}_{k}:=(a_{k}+\frac{\xi}{\delta},a_{k}+\frac{\pi-\xi}{\delta})$ with $a_{k}:=\frac{k\pi+\frac{d+1}{4}\pi}{\delta}-\frac{d}{2}$. Since the length of $\mathcal{I}_{k}$ is $\frac{\pi-2\xi}{\delta}>1$, there exists at least one positive integer in $\mathcal{I}_{k}$ for $k$ being sufficiently large. Taking account of \eqref{eq:RFoS.lower.bound.loc.Fourier.sph-I11}, we have that there exists a subsequence $L_{\ell}$ of $\Zp$ such that
  $|I_{1,1}| = (\sin\tfrac{\delta}{2})^{\frac{d-3}{2}}(\cos\tfrac{\delta}{2})^{\frac{d-1}{2}} \Bigl|\sin\left((L_{\ell}+\tfrac{d}{2})\delta-\tfrac{d+1}{4}\pi\right)\Bigr| > c_{d,\delta,\xi} > 0$, $\ell\ge 1$.
This together with \eqref{eq:RFoS.lower.bound.loc.Fourier.sph-I12}, \eqref{eq:RFoS.lower.bound.loc.Fourier.sph-I1}, \eqref{eq:RFoS.lower.bound.loc.Fourier.sph-I2-2} and \eqref{eq:RFoS.lower.bound.loc.Fourier.sph-Vdl} gives
\begin{equation*}
  \left|\Vdl{L_{\ell}}(\mathbf{1};\PT{x})\right|\ge c_{d,\delta}\: L_{\ell}^{\frac{d-3}{2}}.
\end{equation*}
That is, for $\ell\ge1$,
\begin{equation*}
  \normb{\Vdl{L_{\ell}}(\mathbf{1})}{\Lp{p}{d}}\ge c_{d,\delta}\: L_{\ell}^{\frac{d-3}{2}}.
\end{equation*}
\end{proof}

\section{Filtered local convolutions on the sphere}\label{sec:Rieloc.fisph}
This section proves the upper bound of the filtered local convolution on the sphere. The proof relies on the asymptotic expansion of the filtered kernel of Section~\ref{sec:RFoS.AsympJacobiDirichlet}. Recall that the filtered approximation $\Vdh{L,\fil}$ on $\sph{d}$ is a convolution with a filtered kernel $\vdh{L,\fil}(\PT{x}\cdot \PT{y})$, see Definition~\ref{def:disNsph.fil.fil.ker} and \eqref{eq:filter.sph.approx},
\begin{equation*}
  \Vdh{L,\fil}(f;\PT{x}):=\int_{\sph{d}}\vdh{L,\fil}(\PT{x}\cdot \PT{y})f(\PT{y})\IntDiff{y},\quad f\in\Lp{p}{d},\;\PT{x}\in \sph{d}.
\end{equation*}

Since the filtered convolution kernel $\vdh{L,\fil}(t)$, $-1\le t\le1$, is a constant multiple of the filtered Jacobi kernel $\vabh[(\frac{d-2}{2},\frac{d-2}{2})]{L,\fil}(1,t)$, see Lemma~\ref{lm:RFoS.vdh.vabh}, we are able to use the asymptotic expansion of the latter to prove the upper bound of $\Vdhl{L,\fil}(f)$.

\begin{theorem}\label{thm:RielocFisph.UB filtered sph Lp} Let $d\ge2$, $\kappa\in \Zp$, $1\le p\le \infty$, $0<\delta<\pi$. Let $\fil$ be a filter such that $\fil$ is constant on $[0,1]$ and $\supp \fil\subseteq [0,2]$ and\\
(i) $\fil\in \CkR$;\\
(ii) $\fil\big|_{[1,2]}\in \Ck{\kappa+3}{[1,2]}$.\\
Then, for $f\in \Lp{p}{d}$ and $L\in\Zp$,
\begin{equation*}
  \normb{\Vdhl{L,\fil}(f)}{\Lp{p}{d}}\le
  c\:L^{-(\kappa-\frac{d}{2}+\frac{3}{2})}\left(L^{-1}\norm{f}{\Lp{p}{d}}+\:\modub{f,L^{-\frac{1}{2}}}\right),
\end{equation*}
where the constant $c$ depends only on $d$, $\fil$, $\kappa$, $\delta$ and $p$.
\end{theorem}
Using similar argument to the Remarks following Theorem~\ref{thm:RFoS.upper.bound.local.conv.sphere} and Corollary~\ref{cor:RFoS.UpperBDSph-2}, we obtain the following upper bound of $\Vdhl{L,\fil}(f)$ for a smoother function $f$ on $\sph{d}$.
\begin{corollary}\label{cor:RielocFisph.UB filtered Sphere Lp-Wp} Let $d\ge2$, $\kappa\in \Zp$, $1\le p\le \infty$, $0<\delta<\pi$. Let $\fil$ be a filter such that $\fil$ is constant on $[0,1]$ and $\supp \fil\subseteq [0,2]$ and \\
(i) $\fil\in \CkR$;\\
(ii) $\fil\big|_{[1,2]}\in \Ck{\kappa+3}{[1,2]}$.\\
Then, for $f\in \sob{p}{s}{d}$, $s\ge2$, and $L\in\Zp$,
\begin{equation*}
  \normb{\Vdhl{L,\fil}(f)}{\Lp{p}{d}}\le
  c\:L^{-(\kappa-\frac{d}{2}+\frac{5}{2})}\norm{f}{\sob{p}{s}{d}},
\end{equation*}
where the constant $c$ depends only on $d$, $\fil$, $\kappa$, $\delta$, $p$ and $s$.
\end{corollary}
\begin{remark}
Compared to Theorem~\ref{thm:RFoS.upper.bound.local.conv.sphere} and Corollary~\ref{cor:RFoS.UpperBDSph-2}, Theorem~\ref{thm:RielocFisph.UB filtered sph Lp} and Corollary~\ref{cor:RielocFisph.UB filtered Sphere Lp-Wp} show that the (Riemann) localisation of the Fourier convolution is improved by filtering the Fourier coefficients and that the convergence rate of the filtered local convolution depends on the smoothness of the filter function.
\end{remark}
The commutativity between the translation and Laplace-Beltrami operator implies the upper bound of the Sobolev norm of the filtered local convolution, as follows.
\begin{theorem}\label{thm:RielocFisph.UB filtered Sphere Wp-Wp} Let $d\ge2$, $s\ge0$, $\kappa\in \Zp$, $1\le p\le \infty$, $0<\delta<\pi$, $L\in\Zp$. Let $\fil$ be a filter such that $\fil$ is constant on $[0,1]$ and $\supp \fil\subseteq [0,2]$ and\\
(i) $\fil\in \CkR$;\\
(ii) $\fil\big|_{[1,2]}\in \Ck{\kappa+3}{[1,2]}$.\\
Then, for $f\in \sob{p}{s}{d}$,
\begin{equation*}
  \normb{\Vdhl{L,\fil}(f)}{\sob{p}{s}{d}}\le
  c\:L^{-(\kappa-\frac{d}{2}+\frac{3}{2})}\left(L^{-1}\norm{f}{\sob{p}{s}{d}}+\:\modub[\sob{p}{s}{d}]{f,L^{-\frac{1}{2}}}\right),
\end{equation*}
and for $f\in \sob{p}{s+2}{d}$,
\begin{equation*}
  \normb{\Vdhl{L,\fil}(f)}{\sob{p}{s}{d}}\le
  c\:L^{-(\kappa-\frac{d}{2}+\frac{5}{2})}\left(\norm{f}{\sob{p}{s}{d}}+\norm{\LBo f}{\sob{p}{s}{d}}\right),
\end{equation*}
where the constants $c$ depend only on $d$, $\fil$, $\kappa$, $\delta$, $p$ and $s$.
\end{theorem}

We only prove Theorem~\ref{thm:RielocFisph.UB filtered sph Lp}. The proof of Theorem~\ref{thm:RielocFisph.UB filtered Sphere Wp-Wp} is similar to the proofs of Theorem~\ref{thm:RielocFisph.UB filtered sph Lp} and Corollary~\ref{cor:RielocFisph.UB filtered Sphere Lp-Wp}.
\begin{proof}[Proof of Theorem~\ref{thm:RielocFisph.UB filtered sph Lp}.]
Only the proof for the case $0 < \delta < \frac{\pi}{2}$ is given, as it also implies the case $\frac{\pi}{2} < \delta < \pi$.
For $\PT{x}\in \sph{d}$, by \eqref{eq:RFoS.integral.zonal.via.translation},
\begin{equation*}
    \Vdhl{L,\fil}(f;\PT{x})= \int_{\sph{d}\backslash \scap{\PT{x},\delta}}\vdh{L,\fil}(\PT{x}\cdot \PT{y})f(\PT{y})\IntDiff{y}
                   = \frac{|\sph{d-1}|}{|\sph{d}|}\int_{\delta}^{\pi}\vdh{L,\fil}(\cos\theta)\:\Tsph[(d)]{\theta}(f;\PT{x})(\sin\theta)^{d-1} \mathrm{d}\theta.
\end{equation*}
We split the integral
\begin{align*}
    \Vdhl{L,\fil}(f;\PT{x})
    &= \left(\int_{\delta}^{\frac{\pi}{2}}
        +\int_{\frac{\pi}{2}}^{\pi}\right)\Tsph{\theta}(f;\PT{x})\:\vabh[(\frac{d-2}{2},\frac{d-2}{2})]{L,\fil}(1,\cos\theta)\:(\sin\theta)^{d-1}\:\IntD{\theta}\nonumber\\[2mm]
    &=: \bigl(I_{1}(f;\PT{x})+I_{2}(f;\PT{x})\bigr).
\end{align*}

We apply Theorem~\ref{thm:fiJ.filtered.kernel.asymp-2} with $\alpha=\beta:=(d-2)/2$ to estimate $I_{1}$. Using the notation of Theorem~\ref{thm:fiJ.filtered.kernel.asymp-2}, let
\begin{equation}\label{eq:RielocFisph.thm:UpperBDSph-11}
    \widetilde{m}_{i}(\theta):= C^{(1)}_{\frac{d-2}{2},\frac{d-2}{2},\kappa+3}(\theta)\:\afiJcbu{i}(\theta)\:(\sin\theta)^{d-1},\quad i=1,2,3,4.
\end{equation}
Then
\begin{align}\label{eq:RielocFisph.thm:UpperBDSph-5}
  I_{1}(f;\PT{x})
  &= \int_{\delta}^{\frac{\pi}{2}}\Tsph{\theta}(f;\PT{x})\:\vabh[(\frac{d-2}{2},\frac{d-2}{2})]{L,\fil}(1,\cos\theta)\:(\sin\theta)^{d-1}\IntD{\theta}\notag\\
       &= \int_{\delta}^{\frac{\pi}{2}}\Tsph{\theta}(f;\PT{x})
       \frac{L^{-(\kappa-\frac{d}{2}+\frac{3}{2})}}{2^{\kappa+3}(\kappa+1)!}\brb{\widetilde{m}_{1}(\theta)\cos\phi_{L}(\theta)+ \widetilde{m}_{2}(\theta)\sin\phi_{L}(\theta)
    +\widetilde{m}_{3}(\theta)\cos\overline{\phi}_{L}(\theta) \notag\\
  &\quad +\widetilde{m}_{4}(\theta)\sin\overline{\phi}_{L}(\theta) +(\sin\theta)^{-1}\:\bigo{d,\fil,\kappa}{L^{-1}}}\IntD{\theta}\notag\\
  &= \frac{L^{-(\kappa-\frac{d}{2}+\frac{3}{2})}}{2^{\kappa+3}(\kappa+1)!}\biggl[\int_{\delta}^{\frac{\pi}{2}} \Bigl(\Tsph{\theta}(f;\PT{x})\:\widetilde{m}_{1}(\theta)\cos\phi_{L}(\theta)+ \Tsph{\theta}(f;\PT{x})\:\widetilde{m}_{2}(\theta)\sin\phi_{L}(\theta)
  \notag\\
  &\quad+\Tsph{\theta}(f;\PT{x})\:\widetilde{m}_{3}(\theta)\cos\overline{\phi}_{L}(\theta)
  +\Tsph{\theta}(f;\PT{x})\:\widetilde{m}_{4}(\theta)\sin\overline{\phi}_{L}(\theta)\Bigr)\IntD{\theta}\notag\\
  &\quad+\norm{f}{\Lp{1}{d}}\:\bigo{d,\fil,\kappa,\delta}{L^{-1}}\biggr]\notag\\
       &=: \frac{L^{-(\kappa-\frac{d}{2}+\frac{3}{2})}}{2^{\kappa+3}(\kappa+1)!}\Bigl(I_{1,1}(f;\PT{x})+I_{1,2}(f;\PT{x})+I_{1,3}(f;\PT{x})+I_{1,4}(f;\PT{x})\notag\\
       &\hspace{5mm}
       +\norm{f}{\Lp{1}{d}}\:\bigo{d,\fil,\kappa,\delta}{L^{-1}}\Bigr),
\end{align}
where we used \eqref{eq:RFoS.Tf.UB.f.L1}.

Similar to the proof of \eqref{eq:I11.Lp.nrm}, using Lemma~\ref{lm:RFoS.UB.T.sph} gives, for $i=1,2,3,4$,
\begin{equation*}
  \norm{I_{1,i}(f)}{\Lp{p}{d}}\le c_{d,\fil,\kappa,\delta,p}\: \Bigl(L^{-1}\norm{f}{\Lp{p}{d}} +\modu{f,L^{-1}}\Bigr).
\end{equation*}
This with \eqref{eq:RielocFisph.thm:UpperBDSph-5} gives
\begin{equation*}
  \norm{I_{1}(f)}{\Lp{p}{d}}\le c\: L^{-(\kappa-\frac{d}{2}+\frac{3}{2})}\Bigl(L^{-1}\norm{f}{\Lp{p}{d}}
  +\modu{f,L^{-1}}\Bigr),
\end{equation*}
where the constant $c$ depends only on $d,\fil,\kappa,\delta$ and $p$.

Let $c$ be the constant of Theorem~\ref{thm:fiJ.filtered.kernel.asymp-2} where $\alpha=\beta:=(d-2)/2$. We split the integral of $I_{2}(f;\PT{x})$ into two parts, as follows.
\begin{align*}
  I_{2}(f;\PT{x})
  &= \left(\int_{\frac{\pi}{2}}^{\pi-c L^{-1}}+\int_{\pi-c L^{-1}}^{\pi}\right)\Tsph{\theta}(f;\PT{x})\:\vabh[(\frac{d-2}{2},\frac{d-2}{2})]{L,\fil}(1,\cos\theta)\:(\sin\theta)^{d-1}\IntD{\theta}\notag\\
       &=: I_{2,1}(f;\PT{x})+I_{2,2}(f;\PT{x}).
\end{align*}
For $I_{2,2}(f;\PT{x})$, using \cite[Theorem~3.5]{NaPeWa2006-2} with $\alpha:=(d-2)/2$ gives
\begin{equation*}
\norm{I_{2,2}(f)}{\Lp{p}{d}}
       \le c_{d,p}\int_{\pi-cL^{-1}}^{\pi}\norm{\Tsph{\theta}(f)}{\Lp{p}{d}}\:L^{-\kappa+d-2}(\sin\theta)^{d-1}\IntD{\theta}
       \le
       c_{d,p}\:L^{-(\kappa+2)}\:\norm{f}{\Lp{p}{d}}.
\end{equation*}
For $I_{2,1}(f;\PT{x})$, using Theorem~\ref{thm:fiJ.filtered.kernel.asymp-2} again, cf. \eqref{eq:RielocFisph.thm:UpperBDSph-5},
\begin{align}\label{eq:RielocFisph.thm:UpperBdSph-10}
  I_{2,1}(f;\PT{x})
  &= \int_{\frac{\pi}{2}}^{\pi-cL^{-1}}\Tsph{\theta}(f;\PT{x})\:\vabh[(\frac{d-2}{2},\frac{d-2}{2})]{L,\fil}(1,\cos\theta)\:(\sin\theta)^{d-1}\IntD{\theta}\notag\\
  &= \frac{L^{-(\kappa-\frac{d}{2}+\frac{3}{2})}}{2^{\kappa+3}(\kappa+1)!}\biggl[\int_{\frac{\pi}{2}}^{\pi-cL^{-1}} \Bigl(\Tsph{\theta}(f;\PT{x})\:\widetilde{m}_{1}(\theta)\cos\phi_{L}(\theta)+ \Tsph{\theta}(f;\PT{x})\:\widetilde{m}_{2}(\theta)\sin\phi_{L}(\theta)
  \notag\\
  &\quad +\Tsph{\theta}(f;\PT{x})\:\widetilde{m}_{3}(\theta)\cos\overline{\phi}_{L}(\theta)
  +\Tsph{\theta}(f;\PT{x})\:\widetilde{m}_{4}(\theta)\sin\overline{\phi}_{L}(\theta)\Bigr)\IntD{\theta}\notag\\
  &\quad+\norm{f}{\Lp{1}{d}}\:\bigo{d,\fil,\kappa}{L^{-1}}\biggr]\notag\\
  &=:\frac{L^{-(\kappa-\frac{d}{2}+\frac{3}{2})}}{2^{\kappa+3}(\kappa+1)!}\Bigl(I_{2,1,1}(f;\PT{x})+I_{2,1,2}(f;\PT{x})+I_{2,1,3}(f;\PT{x})+I_{2,1,4}(f;\PT{x})\nonumber\\
  &\quad+\|f\|_{\Lp{1}{d}}\:\bigo{d,\fil,\kappa}{L^{-1}}\Bigr),
\end{align}
where $\widetilde{m}_{i}(\theta)$, $i=1,2,3,4$, are given by \eqref{eq:RielocFisph.thm:UpperBDSph-11} and we used \eqref{eq:RFoS.Tf.UB.f.L1}.

Similar to the derivation of \eqref{eq:I211.Lp.nrm},
\begin{eqnarray*}
    \norm{I_{2,1,i}(f)}{\Lp{p}{d}}
    \le c_{d,\fil,\kappa}\: \Bigl(L^{-1}\norm{f}{\Lp{p}{d}} + \modu{f,L^{-1}}\Bigr),\quad i=1,2,3,4.
\end{eqnarray*}
This with \eqref{eq:RielocFisph.thm:UpperBdSph-10} gives
\begin{equation*}
  \norm{I_{2,1}(f)}{\Lp{p}{d}}\le c_{d,\fil,\kappa,p}\: L^{-(\kappa-\frac{d}{2}+\frac{3}{2})}\Bigl(L^{-1}\norm{f}{\Lp{p}{d}}
  +\modu{f,L^{-1}}\Bigr),
\end{equation*}
thus completing the proof.
\end{proof}

\section{Proofs for Section~\ref{sec:RFoS.AsympJacobiDirichlet}}\label{sec:RFoS.some proofs}
This section proves the lemmas in Section~\ref{sec:RFoS.AsympJacobiDirichlet}.
\subsection{Proofs of lemmas in Section~\ref{subsec:RFoS.asymp.Jacobi}}\label{subsec:RFoS.proof.asymp.Jacobi}
\begin{proof}[Proof of Lemma~\ref{lm:prefisph.vd.vab}]
Using \eqref{eq:RFoS.intro.Dirichlet.kernel.sphere} and \eqref{eq:disNsph.normalised.Gegenbauer} with \eqref{eq:disNsph.dim.sph.harmon} and \eqref{eq:prefiJcb.Jcb.1}, gives
\begin{equation*}
    \vd{L}(t)= \sum_{\ell=0}^{L} Z(d,\ell) \NGegen{\ell}(t)
    = \frac{\Gamma(\frac{d}{2})}{\Gamma{(d)}}\sum_{\ell=0}^{L}
    \frac{(2\ell+d-1)\Gamma(\ell+d-1)}{\Gamma(\ell+\frac{d}{2})}\Jcb[\frac{d-2}{2},\frac{d-2}{2}]{\ell}(t).
\end{equation*}
Using \eqref{eq:fiJ.Dirichlet.Jacobi.kernel} with \eqref{eq:fiJ.Jacobi.normalisation.coe} and \eqref{eq:prefiJcb.Jcb.1} and then gives
\begin{align*}
    \vd{L}(t)
    &= \sqrt{\pi}\frac{\Gamma(\frac{d}{2})}{\Gamma(\frac{d+1}{2})} \sum_{\ell=0}^{L}
        \left(\Jcoe[{(\frac{d-2}{2},\frac{d-2}{2})}]{\ell}\right)^{-1}\Jcb[\frac{d-2}{2},\frac{d-2}{2}]{\ell}(1)\Jcb[\frac{d-2}{2},\frac{d-2}{2}]{\ell}(t)\\
    &= \sqrt{\pi}\frac{\Gamma(\frac{d}{2})}{\Gamma(\frac{d+1}{2})}\:\vab[\frac{d-2}{2},\frac{d-2}{2}]{L}(1,t).
\end{align*}
This gives the first equality of \eqref{eq:RFoS.vd.vab}.
The second equality of \eqref{eq:RFoS.vd.vab} is by \eqref{eq:presphhar.area.sph}.
\end{proof}

\begin{proof}[Proof of Lemma~\ref{lm:RFoS.Jacobi.asymp}] i)~The asymptotic expansion \eqref{eq:RFoS.Jacobi.asymp-1} is from \cite[Eq.~8.21.18, p.~197--198]{Szego1975}. ii)~Recall $\ellshift:=\ell+(\alpha+\beta+1)/2$. For the proof of \eqref{eq:RFoS.Jacobi.asymp.two.term}, we make use of the expansion of the Jacobi polynomial in terms of Bessel functions, see \cite[Main Theorem, p.~980]{FrWo1985}: Given $n\in\Zp$, $\alpha\ge-1/2$, $\alpha-\beta>-2n$ and $\alpha+\beta\ge-1$, for $0<\theta\le\pi-\epsilon$,
\begin{align}\label{eq:RFoS.Jacobi.asymp-3}
\Jcb{\ell}(\cos\theta)&=\frac{\Gamma(\ell+\alpha+1)}{\Gamma(\ell+1)}
    \left(\frac{\theta}{\sin\theta}\right)^{1/2}\left(\sin\tfrac{\theta}{2}\right)^{-\alpha}\left(\cos\tfrac{\theta}{2}\right)^{-\beta}\nonumber\\
    &\qquad
    \times\left(\sum_{k=0}^{n-1}A_{k}(\theta)\frac{J_{\alpha+k}(\ellshift\theta)}{\ellshift^{\alpha+k}}
    +\theta^{\alpha_{1}}\mathcal{O}_{\epsilon}\bigr(\ellshift^{-n}\bigl)\right),
\end{align}
 with arbitrary given $0<\epsilon<\pi$, where $\alpha_{1}:=\alpha+2$ when $n=2$ and $\alpha_{1}:=\alpha$ when $n\neq 2$
and the coefficient $A_{k}(\theta)$ satisfies $A_{k}(\theta)\in C^{\infty}[0,\pi)$ for $1\le k\le n-1$ and, see \cite[Corollary~1, p.~980]{FrWo1985},
\begin{equation}\label{eq:RFoS.Jacobi.asymp.A.0.1}
    A_{0}(\theta):=1,\;\; A_{1}(\theta):=\left(\alpha^{2}-\frac{1}{4}\right)\frac{1-\theta\cot\theta}{2\theta}-\frac{\alpha^{2}-\beta^{2}}{4}\tan\frac{\theta}{2}.
\end{equation}
The asymptotic expansion \cite[Eq.~10.17.1--10.17.3]{NIST:DLMF} and the upper bound \cite[Eq.~10.41.1, Eq.~10.41.4]{NIST:DLMF} of the Bessel function give, for some $c_{0}>0$, $\nu\ge-1/2$ and all $z\ge c_{0}$,
\begin{subequations}\label{eq:RFoS.Bessel.fun.expans}
\begin{align}
  J_{\nu}(z) &= \bigo{}{z^{-\frac{1}{2}}}\label{eq:RFoS.Bessel.fun.upper.bound},\\[1mm]
  J_{\nu}(z) &= \sqrt{\frac{2}{\pi}}\left(z^{-\frac{1}{2}}\cos\omega_{\nu}(z)+\bigo{}{z^{-\frac{3}{2}}}\right),\label{eq:RFoS.Bessel.fun.expan.one.term}\\[1mm]
  J_{\nu}(z) &= \sqrt{\frac{2}{\pi}}\left(z^{-\frac{1}{2}}\cos\omega_{\nu}(z)
  -z^{-\frac{3}{2}}a_{1}(\nu)\sin\omega_{\nu}(z)+\bigo{}{z^{-\frac{5}{2}}}\right),\label{eq:RFoS.Bessel.fun.expan.two.term}
\end{align}
where the constants in the three big $\mathcal{O}$ terms depend only on $\nu$ and $c_{0}$.
\end{subequations}

When $\alpha<1/2$, we take $n=2$ in \eqref{eq:RFoS.Jacobi.asymp-3}. For the Bessel functions $J_{\alpha+k}(\ellshift\theta)$, $k=0,1$, we use \eqref{eq:RFoS.Bessel.fun.expan.two.term} when $k=0$ and \eqref{eq:RFoS.Bessel.fun.expan.one.term} when $k=1$, then for $c\: \ell^{-1}\le \theta\le\pi-\epsilon$ (thus $\ellshift\theta\ge c$),
\begin{align}\label{eq:RFoS.Jacobi.poly.asymp.expan}
    &\Jcb{\ell}(\cos\theta)\nonumber\\[0.1cm]
    &\hspace{1mm}=\frac{\Gamma(\ell+\alpha+1)}{\Gamma(\ell+1)}
    \left(\frac{\theta}{\sin\theta}\right)^{1/2}\left(\sin\tfrac{\theta}{2}\right)^{-\alpha}\left(\cos\tfrac{\theta}{2}\right)^{-\beta}\nonumber\\
    &\quad\times\biggr( A_{0}(\theta)\sqrt{\frac{2}{\pi}}\frac{1}{\ellshift^{\alpha}}\left((\ellshift\theta)^{-\frac{1}{2}}\cos\omega_{\alpha}(\ellshift\theta)
  -(\ellshift\theta)^{-\frac{3}{2}}a_{1}(\alpha)\sin\omega_{\alpha}(\ellshift\theta)
  +\bigo{\alpha}{(\ellshift\theta)^{-\frac{5}{2}}}\right)\nonumber\\
    &\quad+A_{1}(\theta)\sqrt{\frac{2}{\pi}}\frac{1}{\ellshift^{\alpha+1}}\left((\ellshift\theta)^{-\frac{1}{2}}\cos\omega_{\alpha+1}(\ellshift\theta)
  +\bigo{\alpha}{(\ellshift\theta)^{-\frac{3}{2}}}\right)
    +\theta^{\alpha+2}\bigo{\epsilon}{\ellshift^{-2}}\biggl)\nonumber\\[0.1cm]
    &\hspace{1mm}=\frac{\Gamma(\ell+\alpha+1)}{\Gamma(\ell+1)}
    \:\pi^{-\frac{1}{2}}\left(\sin\tfrac{\theta}{2}\right)^{-\alpha-\frac{1}{2}}\left(\cos\tfrac{\theta}{2}\right)^{-\beta-\frac{1}{2}}\ellshift^{-\frac{1}{2}-\alpha}\nonumber\\
    &\quad\times\left(\cos\omega_{\alpha}(\ellshift\theta)+\ellshift^{-1}\FB{\theta}\cos\omega_{\alpha+1}
    (\ellshift\theta)
    +\bigo{\alpha,\beta}{\ellshift^{-2}\theta^{-2}}
    +\bigo{\epsilon}{\ellshift^{-2+(\frac{1}{2}+\alpha)}\theta^{\alpha+\frac{5}{2}}}\right),
\end{align}
where by \eqref{eq:RFoS.Jacobi.asymp.A.0.1}, $\FB{\theta}$ is given by
\begin{align*}
  \FB{\theta}\cos\omega_{\alpha+1}(\ellshift\theta)
  &:= -\frac{A_{0}(\theta)a_{1}(\alpha)}{\theta}\sin\omega_{\alpha}(\ellshift\theta)
    +A_{1}(\theta)\cos\omega_{\alpha+1}(\ellshift\theta)\notag\\[0.1cm]
    &\hspace{1mm}=\left(\frac{\beta^{2}-\alpha^{2}}{4}\tan\frac{\theta}{2}-\frac{4\alpha^{2}-1}{8}\cot\theta\right)
    \cos\omega_{\alpha+1}(\ellshift\theta),
\end{align*}
and \eqref{eq:RFoS.Jacobi.asymp-3} and \eqref{eq:RFoS.Bessel.fun.expans} require $\alpha\ge-1/2$, $\alpha+\beta\ge-1$ and $\alpha-\beta>-4$.
Using \cite[Eq.~5.11.13, Eq.~5.11.15]{NIST:DLMF}, i.e.
\begin{equation}\label{eq:RFoS.Gammma.expan}
  \frac{\Gamma(\ell+u+1)}{\Gamma(\ell+v+1)}=\ell^{u-v}
  \left(1+\frac{(u-v)(u+v+1)}{2}\ell^{-1}+\bigo{u,v}{\ell^{-2}}\right),
\end{equation}
with \eqref{eq:RFoS.Jacobi.poly.asymp.expan} gives
\begin{align}\label{eq:RFoS.Jacobi.asymp-4}
    &\Jcb{\ell}(\cos\theta)
    =\pi^{-\frac{1}{2}}
    \left(\sin\tfrac{\theta}{2}\right)^{-\alpha-\frac{1}{2}}\left(\cos\tfrac{\theta}{2}\right)^{-\beta-\frac{1}{2}}\ellshift^{-\frac{1}{2}}\\
    &\qquad\times\left(\cos\omega_{\alpha}(\ellshift\theta)+\ellshift^{-1}\FA[\ellshift]{1}{\theta}
    +\bigo{\alpha,\beta}{\ell^{-2}\theta^{-2}}
    +\bigo{\epsilon,\alpha,\beta}{\ell^{-2+(\frac{1}{2}+\alpha)}\theta^{\alpha+\frac{5}{2}}}\right),\notag
\end{align}
where
\begin{equation}\label{eq:RFoS.Jacobi.asymp.FA}
  \FA[\ellshift]{1}{\theta}:=\FB{\theta}\cos\omega_{\alpha+1}(\ellshift\theta)-\frac{\alpha\beta}{2}\cos\omega_{\alpha}(\ellshift\theta).
\end{equation}

When $\alpha\ge1/2$, we take $n=n_{\alpha}:=\floor{\frac{1}{2}+\alpha}+2\ge3$ in \eqref{eq:RFoS.Jacobi.asymp-3}. For the Bessel functions $J_{\alpha+k}(\ellshift\theta)$, $0\le k\le n-1$, we use \eqref{eq:RFoS.Bessel.fun.expan.two.term} when $k=0$ and \eqref{eq:RFoS.Bessel.fun.expan.one.term} when $k=1$, and use the upper bound \eqref{eq:RFoS.Bessel.fun.upper.bound} when $2\le k\le n-1$.
Then, for $c\: \ell^{-1}\le \theta\le\pi-\epsilon$, cf. \eqref{eq:RFoS.Jacobi.poly.asymp.expan} and \eqref{eq:RFoS.Jacobi.asymp-4},
\begin{align*}
    \Jcb{\ell}(\cos\theta)
    & =\frac{\Gamma(\ellshift+\frac{\alpha-\beta-1}{2}+1)}{\Gamma(\ellshift+\frac{-\alpha-\beta-1}{2}+1)}
    \left(\frac{\theta}{\sin\theta}\right)^{1/2}\left(\sin\tfrac{\theta}{2}\right)^{-\alpha}\left(\cos\tfrac{\theta}{2}\right)^{-\beta}\\[1mm]
    &\quad\times\biggr(A_{0}(\theta)\sqrt{\frac{2}{\pi}}\frac{1}{\ellshift^{\alpha}}\left((\ellshift\theta)^{-\frac{1}{2}}\cos\omega_{\alpha}(\ellshift\theta)
        -(\ellshift\theta)^{-\frac{3}{2}}a_{1}(\alpha)\sin\omega_{\alpha}(\ellshift\theta)
        +\bigo{\alpha}{(\ellshift\theta)^{-\frac{5}{2}}}\right)\\
    &\quad\qquad+A_{1}(\theta)\sqrt{\frac{2}{\pi}}\frac{1}{\ellshift^{\alpha+1}}\left((\ellshift\theta)^{-\frac{1}{2}}\cos\omega_{\alpha+1}(\ellshift\theta)
    +\bigo{\alpha}{(\ellshift\theta)^{-\frac{3}{2}}}\right)\\
    &\quad\qquad\quad+\sum_{k=2}^{n-1}A_{k}(\theta)\frac{\bigo{\alpha}{(\ellshift\theta)^{-\frac{1}{2}}}}{\ellshift^{\alpha+k}}+\theta^{\alpha}\bigo{\epsilon}{\ellshift^{-n}}\biggl)\\[1mm]
    &=\pi^{-\frac{1}{2}}
    \left(\sin\tfrac{\theta}{2}\right)^{-\alpha-\frac{1}{2}}\left(\cos\tfrac{\theta}{2}\right)^{-\beta-\frac{1}{2}}\ellshift^{-\frac{1}{2}}\\
    &\quad\times\biggl(\cos\omega_{\alpha}(\ellshift\theta) + \ellshift^{-1}\FA[\ellshift]{1}{\theta}
    + \bigo{\alpha,\beta}{\ell^{-2}\theta^{-2}}+ \bigo{\epsilon,\alpha,\beta}{\ell^{-2+\fractional{\alpha+\frac{1}{2}}}\theta^{\alpha+\frac{1}{2}}}\biggr),
\end{align*}
where we used \eqref{eq:RFoS.Gammma.expan} and $\FA[\ellshift]{1}{\theta}$ is given by \eqref{eq:RFoS.Jacobi.asymp.FA}, and in this case \eqref{eq:RFoS.Jacobi.asymp-3} and \eqref{eq:RFoS.Bessel.fun.expans} require $\alpha\ge-1/2$, $\alpha+\beta\ge-1$ and $\alpha-\beta>-2\floor{\frac{1}{2}+\alpha}-4$.
\end{proof}

\subsection{Proofs of lemmas in Section~\ref{subsec:RFoS.gasymp.Dirichlet.kernel}}\label{subsec:RFoS.gproof.asymp.Dirichlet.kernel}
\begin{proof}[Proof of Lemma~\ref{lm:RFoS.Dirichlet.Jacobi.one.term.asymp}.] By \eqref{eq:fiJ.Dirichlet.Jacobi.kernel} and \cite[Eq.~4.5.3, p.~71]{Szego1975}, for $-1\le s\le1$,
\begin{align}\label{eq:RFoS.lm:Dirichlet.Jacobi.two.term.asymp-1}
  \vab{L}(1,s)&=
  \sum_{\ell=0}^{L}\left(\Jcoe{\ell}\right)^{-1}\Jcb{\ell}(1)\Jcb{\ell}(s)\nonumber\\
  &= \frac{1}{2^{\alpha+\beta+1}}\frac{\Gamma(L+\alpha+\beta+2)}{\Gamma(\alpha+1)\Gamma(L+\beta+1)}\Jcb[\alpha+1,\beta]{L}(s).
\end{align}
Then, the estimate in \eqref{eq:RFoS.Dirichlet.Jacobi.one.term.asymp-a} of $\vab{L}(1,\cos\theta)$ for $c^{(1)}L^{-1}\le\theta\le\pi/2$ follows from \eqref{eq:RFoS.Jacobi.asymp-1} of Lemma~\ref{lm:RFoS.Jacobi.asymp}. For $\pi/2<\theta\le\pi-c^{(1)}L^{-1}$,
using
  $\Jcb[\gamma,\eta]{L}(-z)=(-1)^{L}\Jcb[\eta,\gamma]{L}(z)$, $-1\le z\le1,\;\gamma,\eta>-1$,
see \cite[Eq.~4.1.3, p.~59]{Szego1975},
with \eqref{eq:RFoS.lm:Dirichlet.Jacobi.two.term.asymp-1} gives
\begin{equation}\label{eq:RFoS.Dirichlet.Jacobi.one.term.asymp-1}
    \vab{L}(1,\cos\theta)
    =\frac{1}{2^{\alpha+\beta+1}}\frac{\Gamma(L+\alpha+\beta+2)}{\Gamma(\alpha+1)\Gamma(L+\beta+1)}(-1)^{L}\Jcb[\beta,\alpha+1]{L}(\cos\theta'),
\end{equation}
where $\theta':=\pi-\theta$. By \eqref{eq:RFoS.Gammma.expan} with $\ell=\Lshift=L+\tfrac{\alpha+\beta+2}{2}$, $u=\tfrac{\alpha+\beta}{2}$ and $v=\tfrac{-\alpha+\beta-2}{2}$,
\begin{equation}\label{eq:RFoS.Gamma.expan.n.tilde}
 \frac{\Gamma(L+\alpha+\beta+2)}{\Gamma(L+\beta+1)}
 =\Lshift^{\alpha+1}\left(1-\frac{(\alpha+1)\beta}{2}\Lshift^{-1}+\bigo{\alpha,\beta}{L^{-2}}\right)
 =\Lshift^{\alpha+1}\left(1+\bigo{\alpha,\beta}{L^{-1}}\right).
\end{equation}
Applying \eqref{eq:RFoS.Jacobi.asymp-1} to $\Jcb[\beta,\alpha+1]{L}(\cos\theta')$ of \eqref{eq:RFoS.Dirichlet.Jacobi.one.term.asymp-1} and by \eqref{eq:RFoS.Gamma.expan.n.tilde}, we have
\begin{align*}
    &\hspace{0mm}\vab{L}(1,\cos\theta)\\
    &\hspace{2mm}=\frac{2^{-(\alpha+\beta+1)}}{\Gamma(\alpha+1)}\:\Lshift^{\alpha+\frac{1}{2}}\left(1+\bigo{\alpha,\beta}{L^{-1}}\right)
    m_{\beta,\alpha+1}(\theta')\left(\cos\omega_{\beta}(\Lshift\theta')+
  (\sin\theta')^{-1}\bigo{\alpha,\beta}{L^{-1}}\right)\notag\\
    &\hspace{2mm}=\frac{2^{-(\alpha+\beta+1)}}{\Gamma(\alpha+1)}\:\Lshift^{\alpha+\frac{1}{2}}\:
    m_{\beta,\alpha+1}(\theta')\left(\cos\omega_{\beta}(\Lshift\theta')+
  (\sin\theta')^{-1}\bigo{\alpha,\beta}{L^{-1}}\right),
\end{align*}
thus completing the proof.
\end{proof}

\begin{proof}[Proof of Lemma~\ref{lm:RFoS.Dirichlet.Jacobi.two.term.asymp}.]
i) Let $\alpha,\beta>-1/2$ and $\alpha-\beta>-5$, i.e. $(\alpha+1)-\beta>-4$. To estimate $\vab{L}(1,\cos\theta)$, we use \eqref{eq:RFoS.lm:Dirichlet.Jacobi.two.term.asymp-1} and then apply \eqref{eq:RFoS.Jacobi.asymp.two.term} of Lemma~\ref{lm:RFoS.Jacobi.asymp} to $\Jcb[\alpha+1,\beta]{\ell}(\cos\theta)$. Then for $c^{(1)}\ell^{-1}\le\theta\le\pi-\epsilon$, also using \eqref{eq:RFoS.Gamma.expan.n.tilde} ,
\begin{align*}
  &\vab{L}(1,\cos\theta)=\frac{2^{-(\alpha+\beta+1)}}{\Gamma(\alpha+1)}\:\Lshift^{\alpha+1}\:\Bigl[1+\frac{(\alpha+1)\beta}{2}\:\Lshift^{-1}
  +\bigo{\alpha,\beta}{\Lshift^{-2}}\Bigr]\times\Lshift^{-\frac{1}{2}}\:m_{\alpha+1,\beta}(\theta)\\[1mm]
  &\quad\times\Bigr[\cos\omega_{\alpha+1}(\Lshift\theta)+ \Lshift^{-1} F^{(1)}_{\alpha+1,\beta}(\Lshift,\theta)
   +\bigo{\epsilon,\alpha,\beta}{L^{\widehat{u}(\alpha+1)}\theta^{\widehat{\nu}(\alpha+1)}}
   + \bigo{\alpha,\beta}{L^{-2}\theta^{-2}}\Bigr]\\[1mm]
  &=\frac{2^{-(\alpha+\beta+1)}}{\Gamma(\alpha+1)}m_{\alpha+1,\beta}(\theta)\:\Lshift^{\alpha+\frac{1}{2}}\\
  &\quad\times\Bigl[\cos\omega_{\alpha+1}(\Lshift\theta)
  +\Lshift^{-1}\FA[\Lshift]{3}{\theta}
  +\bigo{\epsilon,\alpha,\beta}{L^{\widehat{u}(\alpha+1)}\theta^{\widehat{\nu}(\alpha+1)}}
  + \bigo{\alpha,\beta}{L^{-2}\theta^{-2}}\Bigr],
\end{align*}
where $\widehat{u}(\alpha+1)<-1$ and $\widehat{\nu}(\alpha+1)\ge1$, and by \eqref{subeqs:RFoS.FA.FB},
\begin{equation*}
\FA[\Lshift]{3}{\theta}
    =\frac{(\alpha+1)\beta}{2}\cos\omega_{\alpha+1}(\Lshift\theta)+F^{(1)}_{\alpha+1,\beta}(\Lshift,\theta)
    =\FB[\alpha+1,\beta]{\theta}\cos\omega_{\alpha+2}(\Lshift\theta).
\end{equation*}

\noindent ii) Let $\beta>-1/2$ and $\beta-(\alpha+1)>-4$ (i.e. $\beta-\alpha>-3$) and $\theta':=\pi-\theta \in (c^{(1)}L^{-1},\pi-\epsilon)$. In this case, we make use of \eqref{eq:RFoS.Dirichlet.Jacobi.one.term.asymp-1} and then apply \eqref{eq:RFoS.Jacobi.asymp.two.term} of Lemma~\ref{lm:RFoS.Jacobi.asymp} to $\Jcb[\beta,\alpha+1]{L}(\cos\theta')$.\\
Also by \eqref{eq:RFoS.Gamma.expan.n.tilde}, we have
\begin{align*}
    &\vab{L}(1,\cos\theta)\\[2mm]
    &\quad=\frac{(-1)^{L}2^{-(\alpha+\beta+1)}}{\Gamma(\alpha+1)}\:\Lshift^{\alpha+1}
    \left[1+\frac{(\alpha+1)\beta}{2}\:\Lshift^{-1}+\bigo{\alpha,\beta}{L^{-2}}\right]
    \times\Lshift^{-\frac{1}{2}}\:m_{\beta,\alpha+1}(\theta')\\[1mm]
  &\qquad\times\left[\cos\omega_{\beta}(\Lshift\theta')+\Lshift^{-1} F^{(1)}_{\beta,\alpha+1}(\Lshift,\theta')
  +\bigo{\epsilon,\alpha,\beta}{L^{\widehat{u}(\beta)}{\theta'}^{\widehat{\nu}(\beta)}}
  +\bigo{\alpha,\beta}{L^{-2}{\theta'}^{-2}}\right]\\[2mm]
    &\quad=\frac{(-1)^{L}2^{-(\alpha+\beta+1)}}{\Gamma(\alpha+1)}\:\Lshift^{\alpha+\frac{1}{2}}\:m_{\beta,\alpha+1}(\theta')\\[1mm]
  &\qquad\times\left[\cos\omega_{\beta}(\Lshift\theta')+\Lshift^{-1}F^{(4)}_{\alpha,\beta}(\Lshift,\theta')
  +\bigo{\epsilon,\alpha,\beta}{L^{\widehat{u}(\beta)}{\theta'}^{\widehat{\nu}(\beta)}}
  +\bigo{\alpha,\beta}{L^{-2}{\theta'}^{-2}}\right],
\end{align*}
where by \eqref{subeqs:RFoS.FA.FB},
\begin{equation*}
\FA[\Lshift]{4}{\theta'}
    = F^{(1)}_{\beta,\alpha+1}(\Lshift,\theta')+\frac{(\alpha+1)\beta}{2}\cos\omega_{\beta}(\Lshift\theta')
    = \FB[\beta,\alpha+1]{\theta'}\cos\omega_{\beta+1}(\Lshift\theta').
\end{equation*}
This completes the proof.
\end{proof}
\begin{proof}[Proof of Lemma~\ref{lm:RFoS.Dirichlet.Jacobi.kernel.upper.bound}.] For arbitrary real $\gamma,\eta$, Szeg\H{o} \cite[Theorem~7.32.2, p.~169]{Szego1975} shows
\begin{equation}\label{eq:RFoS.upperJacobi-1}
  \Jcb[\gamma,\eta]{L}(\cos\theta)=\bigo{}{L^{\gamma}},\quad 0\le\theta\le c L^{-1},
\end{equation}
where the constant depends only on $\gamma$ and $\eta$.
The upper bound of \eqref{eq:RFoS.Dirichlet.Jacobi.upper.bound-a} follows from \eqref{eq:RFoS.lm:Dirichlet.Jacobi.two.term.asymp-1} and \eqref{eq:RFoS.upperJacobi-1}, and
\eqref{eq:RFoS.Dirichlet.Jacobi.upper.bound-b} is proved by \eqref{eq:RFoS.Dirichlet.Jacobi.one.term.asymp-1} and \eqref{eq:RFoS.upperJacobi-1}.
\end{proof}

\subsection{Proofs of Section~\ref{subsec:RFoS.fi.ker.asymp}}\label{subsec:proof.fi.ker.asymp}
In this section we prove the asymptotic expansion for the filtered Jacobi kernel in Theorem~\ref{thm:fiJ.filtered.kernel.asymp-2}.

For a sequence $\{u_{\ell} | \:\ell\in \Nz\}$, let $\FDiff{1}\: u_{\ell}:=\FDiff{1} (u_{\ell}):=u_{\ell}-u_{\ell+1}$ be the first order forward difference of $u_{\ell}$. For $s\ge2$, the $s$th order forward difference is then defined recursively by $\FDiff{s} (u_{\ell}):=\FDiff{1}\bigl(\FDiff{s-1}(u_{\ell})\bigr)$. Given $L\in \Zp$, we write the $s$th order forward difference of $\fil(\frac{\cdot}{L})$ as
\begin{equation}\label{eq:fiJ.Zs.ell}
  Z_{s}(\ell) := Z_{s}(L;\ell) :=  \FDiff{s}\: \fil\left(\frac{\ell}{L}\right),\quad \ell=0,1,\dots.
\end{equation}
Let $u_{\ell}$, $\nu_{\ell}$ be two sequences of real numbers. Then it is clear that
\begin{equation}\label{eq:fiJ.Delta.uell.nuell}
    \FDiff{1}\: (u_{\ell}\: \nu_{\ell} )= (\FDiff{1}\: u_{\ell})\: \nu_{\ell} + u_{\ell+1}\:(\FDiff{1}\: \nu_{\ell}).
\end{equation}

Given a filter $\fil$ and $\alpha,\beta>-1$, let $A_{k}(T,t)$ for $T,t\ge0$ be defined recursively by
\begin{equation}\label{eq:fiJ.Ak}
  A_{k}(T,t) := \left\{\begin{array}{ll}
   \displaystyle \fil\left(\frac{t}{T}\right) - \fil\left(\frac{t+1}{T}\right), & k=1,\\[0.3cm]
   \displaystyle\frac{\Ak{k-1}(T,t)}{2t+\alpha+\beta+k}-\frac{\Ak{k-1}(T,t+1)}{2(t+1)+\alpha+\beta+k}, & k=2,3,\dots,
  \end{array}\right.
\end{equation}
see \cite[(4.11)--(4.12), p.~372--373]{IvPeXu2010}.

\begin{lemma}\label{lm:fiJ.Ak.rational.represent}
Let $k\in\Zp$ and $\fil$ be a filter. Then for $L-k\le \ell\le 2L$,
\begin{equation}\label{eq:fiJ.est.Ak-2}
    \Ak{k}(L,\ell)
    = \sum_{i=1}^{k}R_{-(2k-1-i)}^{(k)}(\ell) \:\FDiff{i}\:\fil\Bigl(\frac{\ell}{L}\Bigr),
\end{equation}
where $R_{-j}^{(k)}(\ell)$, $k-1\le j\le 2k-2$, is a rational function of $\ell$ with degree\footnote{Let $R(t)$ be a rational polynomial taking the
form $R(t)=p(t)/q(t)$, where $p(t)$ and $q(t)$ are polynomials with
$q\neq0$. The \emph{degree} of $R(t)$ is $\deg(R):=\deg(p)-\deg(q)$.} $\deg R_{-j}^{(k)}\le -j$ and
\begin{equation*}
    R_{-j}^{(k)}(\ell) = \bigo{k}{\ell^{-j}},\quad
    R_{-(k-1)}^{(k)}(\ell)= 2^{-k}\ell^{-(k-1)}+\bigo{\alpha,\beta,k}{\ell^{-k}}.
\end{equation*}
\end{lemma}
\begin{proof} By definition in \eqref{eq:fiJ.Ak}, letting $r:=\alpha+\beta$ for simplicity,
\begin{align*}
    \Ak{k}(L,\ell)
    &= \left(\frac{\Ak{k-1}(L,\ell)}{2\ell+r+k}-\frac{\Ak{k-1}(L,\ell)}{2(\ell+1)+r+k}\right)
        +\left(\frac{\Ak{k-1}(L,\ell)}{2(\ell+1)+r+k}-\frac{\Ak{k-1}(L,\ell+1)}{2(\ell+1)+r+k}\right)\\[1mm]
    &= \frac{1}{2\ell+r+k+2}\left(\frac{2}{2\ell+r+k}+\FDiff{1}\right) \Ak{k-1}(L,\ell)\\[1mm]
    &=: \ddelta_{k,\ell} \bigl(\Ak{k-1}(L,\ell)\bigr),\quad k\ge2.
\end{align*}
In addition, let $\ddelta_{1,\ell}:=\FDiff{1}$. Then for $k\ge1$,
\begin{equation}\label{eq:fiJ.Ak.deltak}
    \Ak{k}(L,\ell)
    = \ddelta_{k,\ell}\cdots\ddelta_{1,\ell}\left(\fil\Bigl(\frac{\ell}{L}\Bigl)\right).
\end{equation}
Using induction with \eqref{eq:fiJ.Ak.deltak} and \eqref{eq:fiJ.Delta.uell.nuell} gives \eqref{eq:fiJ.est.Ak-2}.
\end{proof}

For a filter $\fil$ satisfying Definition~\ref{def:disNsph.fil.fil.ker}, the asymptotic expansion of the filtered kernel $\vdh{L,\fil}$ depends on the following estimates of $\Ak{k}(L,\ell)$.

\begin{lemma}\label{lm:fiJ.Ak.estimate-a}
Let $r,L\in\Zp$, $1\le k\le r$. Let $\fil$ be a filter satisfying\\
(i)~~$\fil|_{(1,2)}\in \Ck{r}{1,2}$;\\
(ii)~~$\fil^{(i)}$ be bounded in $(1,2)$, $0\le i\le r$.\\
Then,
\begin{equation}\label{eq:fiJ.Ak.UB}
  \Ak{k}(L,\ell)=\bigo{}{L^{-(2k-1)}}, \quad L+1\le \ell\le 2L-k-1,
\end{equation}
where the constant in the big $\mathcal{O}$ term depends only on $k$, $\fil$ and $r$.
\end{lemma}
\begin{proof} The proof uses Lemma~\ref{lm:fiJ.Ak.rational.represent} and the upper bound on $\FDiff{i}\:\fil\Bigl(\frac{\ell}{L}\Bigr)$.
For $\fil\in \CkR[r]$ and $0\le i\le k\le r$, we have by induction the following integral representation of the finite difference
\begin{equation*}
\FDiff{i}\:\fil\Bigl(\frac{\ell}{L}\Bigr) = \int_{0}^{\frac{1}{L}}\IntD{u_{1}}\cdots \int_{0}^{\frac{1}{L}}\fil^{(i)}\left(\frac{\ell}{L}+u_{1}+\cdots+u_{i}\right)\IntD{u_{i}}.
\end{equation*}
Since $\fil^{(i)}$ is bounded in $(1,2)$, for $L+1\le \ell\le 2L-k-1$,
\begin{equation*}
\Bigl|\FDiff{i}\:\fil\Bigl(\frac{\ell}{L}\Bigr)\Bigr|
\le c_{i,\fil}\: L^{-i}.
\end{equation*}
This together with Lemma~\ref{lm:fiJ.Ak.rational.represent} gives \eqref{eq:fiJ.Ak.UB}.
\end{proof}
For $\ell$ near $L$ or $2L$, $\Ak{k}(L,\ell)$ has the following asymptotic expansions.
\begin{lemma}\label{lm:fiJ.Ak.estimate-c}
Let $\fis,k,L\in\Zp$. Let $\fil$ be a filter such that $\fil$ is constant on $[0,1]$ and $\supp g\subset [0,2]$ and\\
(i)~~$\fil\in \CkR$;\\
(ii)~~$\fil|_{[1,2]}\in \Ck{\fis+1}{[1,2]}$.\\
(iii)~~$\fil|_{(1,2)}\in \Ck{\fis+2}{1,2}$ and $\fil^{(\fis+2)}$ is bounded on $(1,2)$.\\
Then for $L+1-k\le \ell\le L$,
\begin{equation*}
  \Ak{k}(L,\ell) =
  L^{-(\fis+k)}\frac{\fil^{(\fis+1)}(1+)}{2^{k}(\fis+1)!}\lambda_{L-\ell,k}^{\fis} + \bigo{}{L^{-(\fis+k+1)}},
\end{equation*}
and for $2L-k\le \ell\le 2L-1$,
\begin{equation*}
  \Ak{k}(L,\ell) =
  L^{-(\fis+k)}\frac{\fil^{(\fis+1)}(2-)}{2^{2k-1}(\fis+1)!}\overline{\lambda}_{2L-\ell-1,k}^{\fis} + \bigo{}{L^{-(\fis+k+1)}},
\end{equation*}
where the constants in the big $\mathcal{O}$ terms depend only on $k$, $\fis$ and $\fil$, and $\lambda_{\nu,s}^{\fis}$ and $\overline{\lambda}_{\nu,s}^{\fis}$ are given by \eqref{subeqs:fiJ.lambda.nu.s}.
\end{lemma}
\begin{proof}
Given $j\in\Zp$, since $\fil|_{[1,2]}\in \Ck{(\fis+1)}{[1,2]}$ and $\fil^{(\fis+2)}|_{(1,2)}$ is bounded in $(1,2)$, then for $\ell\in[L+1,L+k]$, letting $r_{\ell}:=\ell-L$,
\begin{align*}
  \fil\left(\frac{\ell}{L}\right)
  &= \fil\left(1+\frac{r_{\ell}}{L}\right)\\
  &= \fil(1)+\cdots+\frac{\fil^{(\fis)}(1+)}{(\fis+1)!}\left(\frac{r_{\ell}}{L}\right)^{\fis}+\frac{\fil^{(\fis+1)}(1+)}{(\fis+1)!}\left(\frac{r_{\ell}}{L}\right)^{\fis+1}+
  \bigo{k,\fis,\fil}{L^{-(\fis+2)}}.
\end{align*}
Since $\fil\in \CkR$ and $\fil(\cdot)$ is constant on $[0,1]$, $\fil^{(i)}(1+)=0$ for $1\le i\le \fis$. Thus,
\begin{equation}\label{eq:fiJ.fil.Jcb.expan.L}
  \fil\left(\frac{\ell}{L}\right)
   = \fil(1)+ \frac{\fil^{(\fis+1)}(1+)}{(\fis+1)!}\left(\frac{r_{\ell}}{L}\right)^{\fis+1}+\bigo{k,\fis,\fil}{L^{-(\fis+2)}}.
\end{equation}
This gives that for $0\le \ell\le L+k-1$, using the notation in \eqref{eq:fiJ.Zs.ell},
\begin{equation}\label{eq:fiJ.Delta.fil.1}
Z_{1}(\ell):=\arrowDelta{} \fil\left(\frac{\ell}{L}\right) = \fil\left(\frac{\ell}{L}\right) - \fil\left(\frac{\ell+1}{L}\right) =: \coZa{\ell,\fis}\:L^{-(\fis+1)}+\bigo{k,\fis,\fil}{L^{-(\fis+2)}},
\end{equation}
where
\begin{align}\label{eq:fiJ.H.ell.kappa.1}
  \coZa{\ell,\fis} := \left\{
  \begin{array}{ll}
  0,& 0\le\ell\le L-1,\\[1mm]
  \displaystyle-\frac{\fil^{(\fis+1)}(1+)}{(\fis+1)!},& \ell = L,\\[3mm]
  \displaystyle\frac{\fil^{(\fis+1)}(1+)}{(\fis+1)!}\left((r_{\ell})^{\fis+1} - (r_{\ell+1})^{\fis+1}\right),&
  L+1\le\ell\le L+k-1.
  \end{array}
  \right.
\end{align}
For $L-k+1\le\ell\le L$, using \eqref{eq:fiJ.Delta.fil.1},
\begin{align}\label{eq:fiJ.Delta.fil.Zk}
  Z_{k}(\ell):=\arrowDelta{k} \fil\left(\frac{\ell}{L}\right)
  &= \arrowDelta{k-1}\left(\arrowDelta{} \fil\left(\frac{\ell}{L}\right)\right)
  =\sum_{j=0}^{k-1}{k-1\choose j}(-1)^{j}\:Z_{1}(\ell+j)\notag\\
  &= L^{-(\fis+1)}\sum_{j=0}^{k-1}{k-1\choose j}(-1)^{j}\:\coZa{\ell+j,\fis}+\bigo{k,\fis,\fil}{L^{-(\fis+2)}}.
\end{align}
Also, for $0\le\nu\le k-1$, using \eqref{eq:fiJ.H.ell.kappa.1}, letting ${k\choose j}:=0$ for $k<j$,
\begin{align*}
  &\sum_{j=0}^{k-1}{k-1\choose j}(-1)^{j}\coZa{L-\nu+j,\kappa}\notag\\
  &\qquad=\sum_{j=\nu}^{k-1}{k-1\choose j}(-1)^{j}\coZa{L-\nu+j,\kappa}\notag\\
  &\qquad = \frac{\fil^{(\fis+1)}(1+)}{(\fis+1)!}\sum_{j=0}^{k-\nu-1}{k-1\choose j+\nu}(-1)^{j+\nu}
  \left((r_{L+j})^{\fis+1}-(r_{L+j+1})^{\fis+1}\right)\notag\\
  &\qquad = \frac{\fil^{(\fis+1)}(1+)}{(\fis+1)!}\sum_{j=1}^{k-\nu}\left[{k-1\choose j+\nu}+{k-1\choose j+\nu-1}\right](-1)^{j+\nu} (r_{L+j})^{\fis+1}\notag\\
  &\qquad = \frac{\fil^{(\fis+1)}(1+)}{(\fis+1)!}\sum_{j=\nu+1}^{k}{k\choose j}(-1)^{j}(j-\nu)^{\fis+1}\notag\\
  &\qquad = \frac{\fil^{(\fis+1)}(1+)}{(\fis+1)!}\lambda_{\nu,k}^{\fis},
\end{align*}
where $\lambda_{\nu,k}^{\fis}$ is given by \eqref{eq:lambda.nu.s-a} and the second and fourth equations used the transform $j'=j+\nu$.
This with \eqref{eq:fiJ.Delta.fil.Zk} gives, for $0\le\nu\le k-1$,
\begin{align}\label{eq:fiJ.finite.difference.fil-c}
  Z_{k}(L-\nu)
  &= L^{-(\fis+1)}\sum_{j=0}^{k-1}{k-1\choose j}(-1)^{j}\coZa{L-\nu+j,\kappa} + \bigo{k,\fis,\fil}{L^{-(\fis+2)}}\notag\\
  &= L^{-(\fis+1)} \frac{\fil^{(\fis+1)}(1+)}{(\fis+1)!}\:\lambda_{\nu,k}^{\fis} +\bigo{k,\fis,\fil}{L^{-(\fis+2)}}.
\end{align}

On the other hand, for $2L-k\le\ell\le 2L-1$, let $r'_{\ell}:=\ell-2L$. In a similar way to the derivation of \eqref{eq:fiJ.fil.Jcb.expan.L}, we can prove
\begin{equation*}
    \fil\left(\frac{\ell}{L}\right) = \frac{\fil^{(\fis+1)}(2-)}{(\fis+1)!} \left(\frac{r'_{\ell}}{L}\right)^{\fis+1} + \bigo{k,\fis,\fil}{L^{-(\fis+2)}}.
\end{equation*}
Then, for $\ell\ge2L-k$,
\begin{equation}\label{eq:asymp.Delta.fil}
    Z_{1}(\ell):=\arrowDelta{} \fil\left(\frac{\ell}{L}\right)=\coZb{\ell,\fis}\:L^{-(\fis+1)}+\bigo{k,\fis,\fil}{L^{-(\fis+2)}},
\end{equation}
where
\begin{equation}\label{eq:coZb}
     \coZb{\ell,\fis} :=
    L^{-(\fis+1)}\times
    \begin{cases}
    \displaystyle\frac{\fil^{(\fis+1)}(2-)}{(\fis+1)!}\left((r'_{\ell})^{\fis+1} - (r'_{\ell+1})^{\fis+1}\right), & 2L-k\le \ell \le 2L-2,\\[2mm]
    \displaystyle \frac{\fil^{(\fis+1)}(2-)}{(\fis+1)!}(-1)^{\fis+1}, & \ell = 2L-1,\\[2mm]
    0, & \ell \ge 2L.
    \end{cases}
\end{equation}
For $0\le \nu\le k-1$, using \eqref{eq:coZb},
\begin{align}\label{eq:sum.coZb}
    &\sum_{j=0}^{k-1}{k-1 \choose j} (-1)^{j}\: \coZb{2L-1-\nu+j,\fis}\notag\\
    &\qquad= \frac{\fil^{(\fis+1)}(2-)}{(\fis+1)!}\sum_{j=0}^{\nu}{k-1\choose j}(-1)^{j}
\left((r'_{2L-1-\nu+j})^{(\fis+1)} - (r'_{2L-\nu+j})^{(\fis+1)}\right)\notag\\
    &\qquad= \frac{\fil^{(\fis+1)}(2-)}{(\fis+1)!} \sum_{j=0}^{\nu} \left[{k-1 \choose j} + {k-1 \choose j-1}\right](-1)^{j}(r'_{2L-1-\nu+j})^{\fis+1}\notag\\
    &\qquad= \frac{\fil^{(\fis+1)}(2-)}{(\fis+1)!} \sum_{j=0}^{\nu} {k \choose j} (-1)^{j}(j-\nu-1)^{\fis+1}\notag\\
    &\qquad= \frac{\fil^{(\fis+1)}(2-)}{(\fis+1)!}\overline{\lambda}_{\nu,k}^{\fis},
\end{align}
where $\overline{\lambda}_{\nu,k}^{\fis}$ is given by \eqref{eq:lambda.nu.s-b} and we let ${k-1\choose -1}:=0$. Similar to the derivation of \eqref{eq:fiJ.Delta.fil.Zk} and \eqref{eq:fiJ.finite.difference.fil-c}, we then obtain by \eqref{eq:sum.coZb} the asymptotic estimate of $\arrowDelta[\ell]{k}\fil(\frac{\ell}{L})$ for $\ell$ near $2L$: for $0\le\nu\le k-1$,
\begin{align}\label{eq:fiJ.finite.difference.fil-d}
  Z_{k}(2L-1-\nu)
  &= L^{-(\fis+1)} \sum_{j=0}^{k-1}{k-1 \choose j} (-1)^{j}\: \coZb{2L-1-\nu+j,\fis} + \bigo{k,\fis,\fil}{L^{-(\fis+2)}}\notag\\
  &=  L^{-(\fis+1)}\:\frac{\fil^{(\fis+1)}(2-)}{(\fis+1)!}\:\overline{\lambda}_{\nu,k}^{\fis}+\bigo{k,\fis,\fil}{L^{-(\fis+2)}}.
\end{align}

Similar to the first line of \eqref{eq:fiJ.Delta.fil.Zk}, for $i\in\Zp$,
\begin{equation*}
    Z_{i}(\ell)  =\sum_{j=0}^{i-1}{i-1\choose j}(-1)^{j}\:Z_{1}(\ell+j).
\end{equation*}
This with \eqref{eq:fiJ.Delta.fil.1} and \eqref{eq:asymp.Delta.fil} give, for $1\le i\le k$ and $\ell\in [L-k+1,L]\cup[2L-k,2L-1]$,
\begin{equation}\label{eq:Zi.UB}
    Z_{i}(\ell) = \bigo{k,\fis,\fil}{L^{-(\fis+1)}}.
\end{equation}

Now, for $L-k+1\le \ell\le L$, by \eqref{eq:fiJ.finite.difference.fil-c}, the summand $R_{-(2k-1-i)}^{(k)}(\ell) \:\FDiff{i}\:\fil\Bigl(\frac{\ell}{L}\Bigr)$ when $i=k$ in $\eqref{eq:fiJ.est.Ak-2}$ has a lower order than other terms.
We thus split the sum in $\eqref{eq:fiJ.est.Ak-2}$ into two parts: the summand with $i=k$ and the sum of the remaining terms (with $1\le i\le k-1$). Using Lemma~\ref{lm:fiJ.Ak.rational.represent} together with \eqref{eq:fiJ.finite.difference.fil-c} and \eqref{eq:Zi.UB} then gives
\begin{align*}
    \Ak{k}(L,\ell)
    & = R_{-(k-1)}^{(k)}(\ell) \:\FDiff{k}\:\fil\Bigl(\frac{\ell}{L}\Bigr) + \sum_{i=1}^{k-1}R_{-(2k-1-i)}^{(k)}(\ell) \:\FDiff{i}\:\fil\Bigl(\frac{\ell}{L}\Bigr)\\
    & = R_{-(k-1)}^{(k)}(\ell) \:Z_{k}\bigl(L-(L-\ell)\bigr) + \sum_{i=1}^{k-1}R_{-(2k-1-i)}^{(k)}(\ell) \:Z_{i}(\ell)\\
    & = L^{-(\fis+k)}\frac{\fil^{(\fis+1)}(1+)}{2^{k}(\fis+1)!}\lambda_{L-\ell,k}^{\fis} + \bigo{k,\fis,\fil}{L^{-(\fis+k+1)}}.
\end{align*}

Similarly, for $2L-k\le \ell\le 2L-1$, using Lemma~\ref{lm:fiJ.Ak.rational.represent} with \eqref{eq:fiJ.finite.difference.fil-d} and \eqref{eq:Zi.UB} gives
\begin{align*}
    \Ak{k}(L,\ell)
    & = R_{-(k-1)}^{(k)}(\ell) \:Z_{k}\bigl(2L-1-(2L-1-\ell)\bigr) + \sum_{i=1}^{k-1}R_{-(2k-1-i)}^{(k)}(\ell) \:Z_{i}(\ell)\\
    & = L^{-(\fis+k)}\frac{\fil^{(\fis+1)}(2-)}{2^{2k-1}(\fis+1)!}\overline{\lambda}_{2L-\ell-1,k}^{\fis} + \bigo{k,\fis,\fil}{L^{-(\fis+k+1)}},
\end{align*}
thus completing the proof.
\end{proof}

\begin{proof}[Proof of Theorem~\ref{thm:fiJ.filtered.kernel.asymp-2}] From \cite[Eq.~4.5.3,~p.~71]{Szego1975}, for $\ell\ge0$ and $t\in [-1,1]$,
\begin{align*}
  \sum_{j=0}^{\ell}\left(\Jcoe{j}\right)^{-1}\Jcb{j}(1)\Jcb{j}(t)
  &= \sum_{j=0}^{\ell}\frac{2j+\alpha+\beta+1}{2^{\alpha+\beta+1}}\frac{\Gamma(j+\alpha+\beta+1)}{\Gamma(j+\beta+1)\Gamma(\alpha+1)}\Jcb{j}(t)\nonumber\\
  &= \frac{1}{2^{\alpha+\beta+1}}\frac{\Gamma(\ell+\alpha+\beta+2)}{\Gamma(\alpha+1)\Gamma(\ell+\beta+1)}\Jcb[\alpha+1,\beta]{\ell}(t).
\end{align*}
This and repeated use of summation by parts in \eqref{eq:fiJ.Dirichlet.Jacobi.kernel} give
\begin{align}\label{eq:fiJ.filtered.kernel.asymp-a-3}
  \vabh{L,\fil}(1,t) &= \frac{1}{2^{\alpha+\beta+1}\Gamma(\alpha+1)}
  \sum_{\ell=0}^{\infty}\fil\left(\frac{\ell}{L}\right)\frac{(2\ell+\alpha+\beta+1)\Gamma(\ell+\alpha+\beta+1)}{\Gamma(\ell+\beta+1)}P_{\ell}^{(\alpha,\beta)}(t)\nonumber\\
  &= \frac{1}{2^{\alpha+\beta+1}\Gamma(\alpha+1)}\sum_{\ell=0}^{\infty}\Ak{k}(L,\ell)\frac{\Gamma(\ell+\alpha+k+\beta+1)}{\Gamma(\ell+\beta+1)}P_{\ell}^{(\alpha+k,\beta)}(t),
\end{align}
where $\Ak{k}(L,\ell)$ is defined by \eqref{eq:fiJ.Ak}
and since $\fil$ is constant on $[0,1]$ and $\supp \fil = [0,2]$, the support of $\Ak{k}(L,\cdot)$ is $[L-k+1,2L-1]$. Using \eqref{eq:fiJ.filtered.kernel.asymp-a-3} and Lemma~\ref{lm:RFoS.Jacobi.asymp} (adopting its notation) gives
\begin{align}\label{eq:fiJ.filtered.kernel.asymp-a-1}
 &\vabh{L,\fil}(1,\cos\theta)\notag\\
  &\quad= \frac{1}{2^{\alpha+\beta+1}\Gamma(\alpha+1)}\sum_{\ell=0}^{\infty}\Ak{k}(L,\ell)\frac{\Gamma(\ell+\alpha+k+\beta+1)}{\Gamma(\ell+\beta+1)}\nonumber\\
  &\qquad\times \ellshift^{-\frac{1}{2}} \:\pi^{-\frac{1}{2}}
    \left(\sin\tfrac{\theta}{2}\right)^{-(\alpha+k)-\frac{1}{2}}\left(\cos\tfrac{\theta}{2}\right)^{-\beta-\frac{1}{2}}
    \left(\cos\omega_{\alpha+k}(\ellshift\theta)+
    (\sin\theta)^{-1}\bigo{\alpha,\beta,k}{\ell^{-1}}\right)\notag\\[1mm]
    &\quad=
    \frac{\left(\sin\frac{\theta}{2}\right)^{-\alpha-k-\frac{1}{2}}\left(\cos\frac{\theta}{2}\right)^{-\beta-\frac{1}{2}}}{2^{\alpha+\beta+1}\sqrt{\pi}\:\Gamma(\alpha+1)}\nonumber\\
    &\qquad\times\left(\sum_{\ell=L-k+1}^{2L-1}a_{k}(L,\ell)\cos\omega_{\alpha+k}(\ellshift\theta)
    +(\sin\theta)^{-1}\:\bigo{\alpha,\beta,k}{\sum_{\ell=L-k+1}^{2L-1}|a_{k}(L,\ell)|\:\ellshift^{-1}}\right)\nonumber\\[2mm]
    &\quad=: C^{(1)}_{\alpha,\beta,k}(\theta)\left(I_{k,1}+(\sin\theta)^{-1}I_{k,2}\right),
\end{align}
where
\begin{subequations}\label{subeqs:fiJ.ak.C1}
\begin{align}
  \ellshift&:=\ellshift(\alpha+k,\beta):=\ell+\frac{\alpha+k+\beta+1}{2}, \label{eq:fiJ.ellshift.k}\\[1mm]
  a_{k}(L,\ell)
  &:=\displaystyle \Ak{k}(L,\ell)\frac{\Gamma(\ell+\alpha+k+\beta+1)}{\Gamma(\ell+\beta+1)}\:\ellshift^{-\frac{1}{2}},\label{eq:fiJ.ak}\\[1mm]
  C^{(1)}_{\alpha,\beta,k}(\theta)
  &:=
 \frac{\left(\sin\frac{\theta}{2}\right)^{-\alpha-k-\frac{1}{2}}\left(\cos\frac{\theta}{2}\right)^{-\beta-\frac{1}{2}}}{2^{\alpha+\beta+1}\sqrt{\pi}\:\Gamma(\alpha+1)}.
 \label{eq:fiJ.C1}
  \end{align}
\end{subequations}

Now in \eqref{eq:fiJ.filtered.kernel.asymp-a-1} and \eqref{subeqs:fiJ.ak.C1}, letting $k=\fis+3$.
Lemma~\ref{lm:fiJ.Ak.estimate-c} with \eqref{eq:fiJ.ak} and \eqref{eq:fiJ.Gamma.asymp.one.term} gives the asymptotic expansion of $a_{\fis+3}(L,\ell)$ for $\ell$ near $L$ and $2L$, as follows. For $L+1-(\fis+3)\le \ell\le L$,
\begin{subequations}\label{subeqs:fiJ.ak.asymp}
\begin{align}\label{eq:fiJ.ak.asymp-a}
  a_{\fis+3}(L,\ell) =
  L^{-(\fis-\alpha+\frac{1}{2})}\frac{\fil^{(\fis+1)}(1+)}{2^{\fis+3}(\fis+1)!}\lambda_{L-\ell,\fis+3}^{\fis} + \bigo{}{L^{-(\fis-\alpha+\frac{3}{2})}}.
\end{align}
For $2L-(\fis+3)\le \ell\le 2L-1$,
\begin{align}\label{eq:fiJ.ak.asymp-b}
  a_{\fis+3}(L,\ell) =
  L^{-(\fis-\alpha+\frac{1}{2})}\frac{\fil^{(\fis+1)}(2-)}{2^{\fis-\alpha+\frac{5}{2}}(\fis+1)!}\overline{\lambda}_{2L-\ell-1,\fis+3}^{\fis} + \bigo{}{L^{-(\fis-\alpha+\frac{3}{2})}}.
\end{align}
For $L+1\le\ell\le2L-1-(\fis+3)$, using Lemma~\ref{lm:fiJ.Ak.estimate-a} (where we let $r=\fis+3$) with \eqref{eq:fiJ.Gamma.asymp.one.term} gives
\begin{equation}\label{eq:fiJ.ak.asymp-c}
  a_{\fis+3}(L,\ell)=\bigo{}{L^{-(\fis-\alpha+\frac{5}{2})}},
\end{equation}
\end{subequations}
where the constants in the big $\mathcal{O}$'s in \eqref{subeqs:fiJ.ak.asymp} depend only on $\alpha$, $\beta$, $\fil$ and $\fis$.

For $I_{\fis+3,2}$ in \eqref{eq:fiJ.filtered.kernel.asymp-a-1} (where $k=\fis+3$), using \eqref{subeqs:fiJ.ak.asymp} gives
\begin{align}\label{eq:fiJ.asymp.I.kappa+3.2}
    I_{\fis+3,2}
    &= \bigo{}{\sum_{\ell=L-(\fis+2)}^{2L-1}|a_{\fis+3}(L,\ell)\:\ellshift^{-1}|}\notag\\
    &= \left(\sum_{\ell=L-(\fis+2)}^{L}+\sum_{\ell=2L-(\fis+3)}^{2L-1}\right)\bigo{}{L^{-(\fis-\alpha+\frac{1}{2})}\ellshift^{-1}}
    +\sum_{\ell=L+1}^{2L-1-(\fis+3)}\bigo{}{L^{-(\fis-\alpha+\frac{5}{2})}\ellshift^{-1}}\notag\\
    &=\bigo{\alpha,\beta,\fil,\fis}{L^{-(\fis-\alpha+\frac{3}{2})}}.
\end{align}

For $I_{\fis+3,1}$ in \eqref{eq:fiJ.filtered.kernel.asymp-a-1}, using \eqref{subeqs:fiJ.ak.asymp} gives
\begin{align}\label{eq:fiJ.asymp.I.kappa+3.1}
    I_{\fis+3,1}
    &= \left(\sum_{\ell=L-(\fis+2)}^{L}+\sum_{\ell=L+1}^{2L-1-(\fis+3)}+\sum_{\ell=2L-(\fis+3)}^{2L-1}\right)
    a_{\fis+3}(L,\ell)\cos\omega_{\alpha+\fis+3}(\ellshift\theta)\notag\\[1mm]
    &=
    \left(\sum_{\ell=L-(\fis+2)}^{L}+\sum_{\ell=2L-(\fis+3)}^{2L-1}\right)
    a_{\fis+3}(L,\ell)\cos\omega_{\alpha+\fis+3}(\ellshift\theta)+\bigo{\alpha,\beta,\fil,\fis}{L^{-(\fis-\alpha+\frac{3}{2})}}\notag\\[1mm]
    &=
    L^{-(\fis-\alpha+\frac{1}{2})}\:B_{g,\fis}(L) + \bigo{\alpha,\beta,\fil,\fis}{L^{-(\fis-\alpha+\frac{3}{2})}},
\end{align}
where
\begin{equation*}
B_{\fil,\fis}(L) := \left(\frac{\fil^{(\fis+1)}(1+)}{2^{\fis+3}(\fis+1)!}\sum_{\ell=L-(\fis+2)}^{L}\hspace{-5mm}
    \lambda_{L-\ell,\fis+3}^{\fis}
    +\frac{\fil^{(\fis+1)}(2-)}{2^{\fis-\alpha+\frac{5}{2}}(\fis+1)!}\sum_{\ell=2L-(\fis+3)}^{2L-1}
    \hspace{-5mm}\overline{\lambda}_{2L-\ell-1,\fis+3}^{\fis}\right)\cos\omega_{\alpha+\fis+3}(\ellshift\theta).
\end{equation*}
Using the substitution $\ell=L-i$ and $\ellshift[(L-i)](\alpha+\fis+3,\beta)=\Lshift+\tfrac{\fis+2}{2}-i$ (see \eqref{eq:fiJ.ellshift.k}) for the first sum where $\Lshift:=L+\tfrac{\alpha+\beta+2}{2}$,
and using the substitution $\ell=2L-1-i$ and $\ellshift[(2L-1-i)](\alpha+\fis+3,\beta)=\Lshift[2L]-1+\tfrac{\fis+2}{2}-i$ for the second sum where $\Lshift[2L]:=2L+\tfrac{\alpha+\beta+2}{2}$, the above $B_{\fil,\fis}(L)$ then becomes
\begin{align}\label{eq:B.fil.L.a}
B_{\fil,\fis}(L)
  &=\frac{\fil^{(\fis+1)}(1+)}{2^{\fis+3}(\fis+1)!}\sum_{i=0}^{\fis+2}
    \lambda_{i,\fis+3}^{\fis}\cos\omega_{\alpha+\fis+3}\left((\Lshift+\tfrac{\fis+2}{2}-i)\theta\right)\nonumber\\
    &\qquad+\frac{\fil^{(\fis+1)}(2-)}{2^{\fis-\alpha+\frac{5}{2}}(\fis+1)!}\sum_{i=0}^{\fis+2}
    \overline{\lambda}_{i,\fis+3}^{\fis}\cos\omega_{\alpha+\fis+3}\left((\Lshift[2L]-1+\tfrac{\fis+2}{2}-i)\theta\right).
\end{align}
Let $\xi_{1}:=\tfrac{\alpha+\fis+3}{2}\pi+\tfrac{\pi}{4}$ and let $\phi_{L}(\theta):=\omega_{\alpha+\fis+3}((\Lshift+\tfrac{\fis+2}{2})\theta)=\bigl(\Lshift+\tfrac{\fis+2}{2}\bigr)\theta-\xi_{1}$ and $\overline{\phi}_{L}(\theta):=\omega_{\alpha+\fis+3}((\Lshift[2L]-1+\tfrac{\fis+2}{2})\theta)=\bigl(\widetilde{2L}-1+\tfrac{\fis+2}{2}\bigr)\theta-\xi_{1}$, where we used \eqref{eq:RFoS.omega}. Then
\begin{align*}
    \cos\omega_{\alpha+\fis+3}\left((\Lshift+\tfrac{\fis+2}{2}-i)\theta\right) &= \cos\bigl(i\theta\bigr)\cos\phi_{L}(\theta) + \sin\bigl(i\theta\bigr)\sin\phi_{L}(\theta)\\[1mm]
    \cos\omega_{\alpha+\fis+3}\left((\Lshift[2L]-1+\tfrac{\fis+2}{2}-i)\theta\right) &= \cos\bigl(i\theta\bigr)\cos\overline{\phi}_{L}(\theta) +
    \sin\bigl(i\theta\bigr)\sin\overline{\phi}_{L}(\theta).
\end{align*}
This with \eqref{eq:B.fil.L.a} gives
\begin{align*}
B_{\fil,\fis}(L)
  &=\frac{1}{2^{\fis+3}(\fis+1)!}\Bigl(\afiJcbu{1}(\theta)\cos\phi_{L}(\theta)+ \afiJcbu{2}(\theta)\sin\phi_{L}(\theta)  +\afiJcbu{3}(\theta)\cos\overline{\phi}_{L}(\theta)\nonumber\\
  &\qquad+\afiJcbu{4}(\theta)\sin\overline{\phi}_{L}(\theta)\Bigr),
\end{align*}
where
\begin{align*}
  \afiJcbu{1}(\theta) := \fil^{(\fis+1)}(1+)\sum_{i=0}^{\fis+2}\lambda_{i,\fis+3}^{\fis}\cos(i\theta),
  & \quad \afiJcbu{3}(\theta) := 2^{\alpha+\frac{1}{2}}\fil^{(\fis+1)}(2-)\sum_{i=0}^{\fis+2}\overline{\lambda}_{i,\fis+3}^{\fis}\cos(i\theta),\\
  \afiJcbu{2}(\theta) := \fil^{(\fis+1)}(1+)\sum_{i=0}^{\fis+2}\lambda_{i,\fis+3}^{\fis}\sin(i\theta),
  & \quad \afiJcbu{4}(\theta) := 2^{\alpha+\frac{1}{2}}\fil^{(\fis+1)}(2-)\sum_{i=0}^{\fis+2}\overline{\lambda}_{i,\fis+3}^{\fis}\sin(i\theta).
\end{align*}
This together with \eqref{eq:fiJ.asymp.I.kappa+3.1}, \eqref{eq:fiJ.asymp.I.kappa+3.2} and \eqref{eq:fiJ.filtered.kernel.asymp-a-1} gives \eqref{eq:vabh.asymp.expan}.

Since $\cos(\ell\theta)=\mathcal{T}_{\ell}(\cos\theta)$, where $\mathcal{T}_{\ell}(\cdot)$ is the Chebychev polynomial of the first kind of degree $\ell$ with initial coefficient $2^{\ell-1}$ (see e.g. \cite[Section~18.3]{NIST:DLMF}), then $\afiJcbu{1}(\theta)$ is an algebraic polynomial of $\cos\theta$ of degree $\fis+2$ with the initial coefficient
$2^{\fis+1}\fil^{(\fis+1)}(1+)\lambda_{\fis+2,\fis+3}^{\fis}$ $= (-2)^{\fis+1}\fil^{(\fis+1)}(1+)$,
thus completing the proof of the theorem.
\end{proof}

\hspace{1.5cm}

\noindent\textbf{Acknowledgements}~~~ The authors would like to thank Christian Gerhards and Leonardo Colzani for their discussion and comments on the convergence of the Fourier local convolution and the localisation principle.
The authors also thank the anonymous referees for their comments on simplifying the proof of Theorem~\ref{thm:RFoS.upper.bound.local.conv.sphere}.
This research was supported under the Australian Research Council's \emph{Discovery Project} DP120101816.
The first author was supported under the University International Postgraduate Award (UIPA) of UNSW Australia.

\bibliographystyle{abbrv}
\bibliography{Rieloc}
\end{document}